\documentclass[12pt]{article}

\usepackage{amsmath}
\usepackage{amsfonts}
\usepackage{latexsym}
\usepackage{graphicx}
\usepackage{amssymb}
\usepackage{amsthm}
\usepackage[margin=2cm]{geometry}
\usepackage{url}
\usepackage[usenames,dvipsnames]{color}
\usepackage{enumerate}
\usepackage[shortlabels]{enumitem} 
\usepackage[small]{caption}
\usepackage{cite}

\usepackage{tikz}

\usepackage[T1]{fontenc}
\usepackage{lmodern}
\usepackage[utf8]{inputenc}

\usepackage[disable]{todonotes}

\newtheorem{definition}{Definition}
\newtheorem{lemma}[definition]{Lemma}
\newtheorem{proposition}[definition]{Proposition}

\newtheorem{corollary}[definition]{Corollary}

\newtheorem{remark}[definition]{Remark}

\newtheorem{assumption}{Assumption}

\newcommand{\N}{\mathbb{N}}
\newcommand{\Z}{\mathbb{Z}}
\newcommand{\Q}{\mathbb{Q}}
\newcommand{\R}{\mathbb{R}}

\renewcommand{\S}{\mathcal{S}}
\newcommand{\bx}{\mathbf{x}}
\newcommand{\by}{\mathbf{y}}

\newcommand{\ba}{\mathbf{a}}
\newcommand{\bb}{\mathbf{b}}

\newcommand{\Sigmab}{\overline\Sigma}

\begin{document}

\title{On some symmetric multidimensional\\ continued fraction algorithms}

\author{Pierre Arnoux
\thanks{
Equipe Dynamique, Arithmétique et Combinatoire,
Institut de Mathématique de Luminy,
CNRS UPR 9016,
163, avenue de Luminy, Case 907,
13288 MARSEILLE cedex 09,
FRANCE,
\texttt{pierre@pierrearnoux.fr}}
\and
Sébastien Labbé
\thanks{
B\^at. B37 Institut de Math\'ematiques,
Grande Traverse 12,
4000 Li\`ege, Belgium,
\texttt{slabbe@ulg.ac.be}}
}

\date{}

\maketitle

\begin{abstract}
We compute explicitly the density of the invariant measure for the  Reverse algorithm which is absolutely continuous with respect to Lebesgue measure, using a method proposed by Arnoux
and Nogueira. We also apply the same method on the unsorted version of
Brun algorithm and Cassaigne algorithm. We illustrate some experimentations on the domain of the natural
extension of those algorithms. For some other algorithms, which are known to have a unique invariant measure absolutely continuous with respect to Lebesgue measure, the invariant domain found by this method seems to have a fractal boundary, and it is unclear that it is of positive measure.

Natural extension. Fractal.

\end{abstract}


\section{Introduction}\label{sec:intro}

Many continued fraction algorithms, in one or several dimensions, have a unique ergodic measure which is absolutely continuous with respect to Lebesgue measure (the so-called Gauss measure); this measure plays an important role in the dynamic of the algorithm, and it is interesting, when possible, to give an explicit closed formula for the invariant density. 

It is easy to check if a given function $\delta$ is an invariant density for a map $F$, since it must be solution of the functional equation $\delta(x)=\sum_{y; F(y)=x} \frac {\delta(y)}{J_F(y)}$, where $J_F$ is the jacobian of $F$; but it is generally difficult to find an explicit solution to this equation. One method, as proposed by Arnoux and Nogueira \cite{MR1251147}, is to build a geometric model of the natural extension of $F$, since it is often easier to find invariant densities for invertible maps. Another advantage of this method is that it gives a suspension flow on a manifold which can sometimes be linked to other areas of mathematics; for example, for Farey continued fraction, this flow can be recognized as the geodesic flow on the modular surface; for Rauzy induction on interval exchange maps, it is a Teichm\"uller flow.

The basic idea is as follows. A continued fraction algorithm can often be presented as a piecewise-linear map $F$ on the positive cone of $\R^d$ (or a subcone of it), given locally as $\bx\mapsto M^{-1}.\bx$, where $M$ is  a positive matrix. We can try to build a model $\widetilde F$ for the natural extension as a skew product on a cone of $\R^d\times\R^d$, $(\bx, \ba)\mapsto (M^{-1}.\bx, M^{\top}.\ba)$. We can look for a cone on which this map is one-to-one, up to a set of measure 0; contraction arguments (A generalization of Hutchinson's theorem, see \cite{2015_Arnoux_Schmidt}) ensure, in most interesting cases, that such a set is unique, and must exist (but it could be of measure 0).

If such a set $D$ can be found, then $\widetilde F: D\to D$ is by construction a bijection which preserves Lebesgue measure and the scalar product $\bx.\ba$.  Furthermore, the flow $\varphi_t$ defined on $D$ by $\varphi_t(\bx,\ba)=(e^t\bx, e^{-t}\ba)$ also preserves the measure and the scalar product, and commutes with $\widetilde F$. Let  $D_1$ be the subset of $D$ defined by $\bx.\ba=1$, and $\Omega$ be the quotient $D_1/\widetilde F$; the flow $\varphi_t$ projects to the quotient, and we can construct a natural extension of the projective map $f$ associated with $F$ as first return map of the flow to a section.

In this paper, we use this method to find the invariant density for two multidimensional continued fraction algorithms, the Reverse algorithm and the Brun algorithm, and we explain the problems encountered to apply  this method to some other algorithms. Finding the set $D$ is the crucial part of the method. A set $D$ is known for Brun algorithm and is not known for the Hurwitz nor the Jacobi-Perron algorithm. Experimentations seem to show that it has a fractal structure for Jacobi-Perron and Arnoux-Rauzy-Poincaré algorithms.

In Section~\ref{sec:farey}, as an example, we show how to compute the invariant
density for the unsorted Farey map. In Section~\ref{sec:constructing}, we give the
general argument. In Section~\ref{sec:reverse}, we apply it to the Reverse
algorithm, in Section~\ref{sec:cassaigne} to the Cassaigne algorithm and in
Section~\ref{sec:brun}, to the Brun algorithm. In
Section~\ref{sec:experimentations}, we explain how numerical computations can show in some
cases the invariant set, but give more complicated results (fractal sets) in
other cases. In Section~\ref{sec:remarks}, we make some  remarks about sorted and
unsorted algorithms, acceleration and choices of coordinates, to explain why
one can find many variants of the continued fraction algorithms, and we
illustrate that on the classical continued fraction. We also give the explicit
computation of the invariant density for Brun's algorithm in any dimension.

\section{The unsorted Farey map}\label{sec:farey}

Consider the positive cone $\Lambda=\{(x,y)\in\R^2\mid 0<x,0<y\}$ and define
\[
F:(x,y)\mapsto
\begin{cases}
(x,y-x)   & \mbox{if } x<y, \\
(x-y,y)   & \mbox{if } x>y.
\end{cases}
\]

\begin{remark}
The map $F$ is well-defined, except on a set of Lebesgue measure 0. There is no canonical way to define it on the diagonal $x=y$; we could choose a convention, but since we will be interested in the ergodic properties of $F$, this is irrelevant. This will occur in all the examples of the paper, and we will no more remark about it.
\end{remark}

Since the map $F$ is homogeneous of degree 1 (it commutes with positive homotheties), there is an associated projective map $f$, which can be defined on the unit simplex. This projective version of $F$ on $\Delta = \{ (x,y)\in\Lambda \mid
x+y=1\}$ is defined as $f: \bx \mapsto \frac{F(\bx)}{\Vert
F(\bx)\Vert_1}$. If we take the first variable $x$ as coordinate on the unit
simplex, $f$ can be defined as (see Figure~\ref{fig:Fareyt})
\[
\begin{array}{rrcl}
f: & [0,1]  & \to     & [0,1]\\
   & x     & \mapsto & 
\begin{cases}
\frac{x}{1-x}   & \mbox{if } x<\frac{1}{2}, \\
2-\frac{1}{x}   & \mbox{if } x>\frac{1}{2}.
\end{cases}
\end{array}
\]

\begin{figure}[h]
\begin{center}
\includegraphics{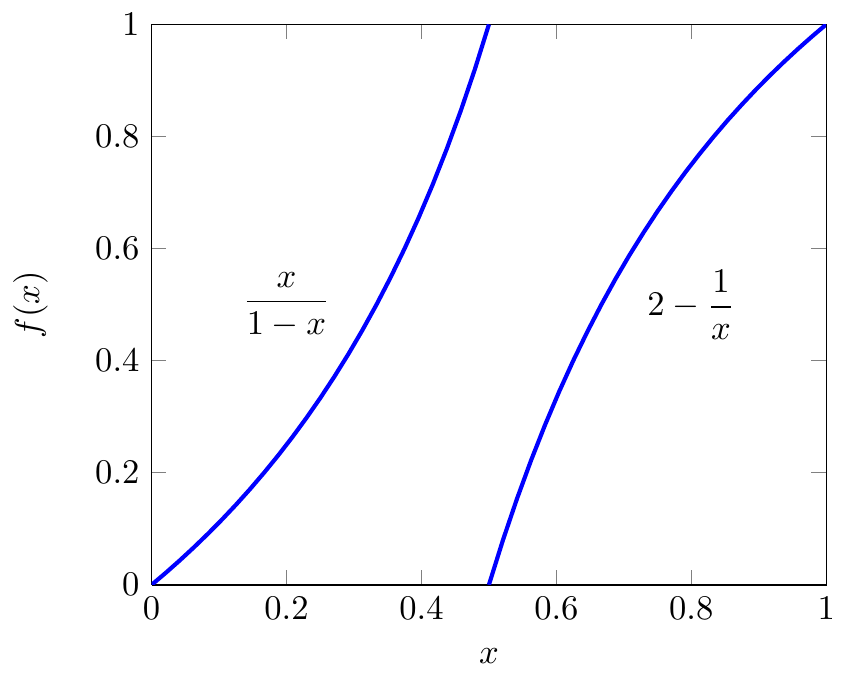}
\end{center}
\caption{
The graph of the unsorted Farey map.
}
\label{fig:Fareyt}
\end{figure}

\begin{remark}
Except for a set of measure 0 (the points $(x,y)$ such that $\frac xy$ is rational), the positive orbit $\{ F^n(x,y)\}$ is defined for all $n\in \N$, and the norm $|x+y|$ is strictly decreasing along the orbit; the same proof shows that the positive orbit $\{f^n(x)\}$ is well defined  for all $n\in \N$ if and only if $x$ is irrational.
\end{remark}
Below we compute an invariant measure of $f$,  using a method that also works
for some multidimensional continued fractions algorithms. Note that neither $F$ nor
$f$ are bijection. For instance, we have $F(2,5)=F(5,3)=(2,3)$:
\begin{center}
\includegraphics{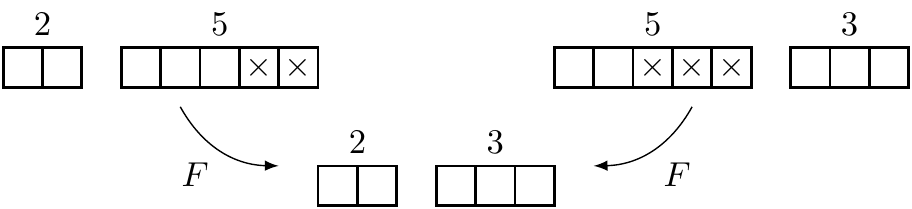}
\end{center}
One way to transform $F$ into a bijection is to stack the removed part and
pile up the result in the following way, as done in \cite{MR1251147}:
\begin{center}
\includegraphics{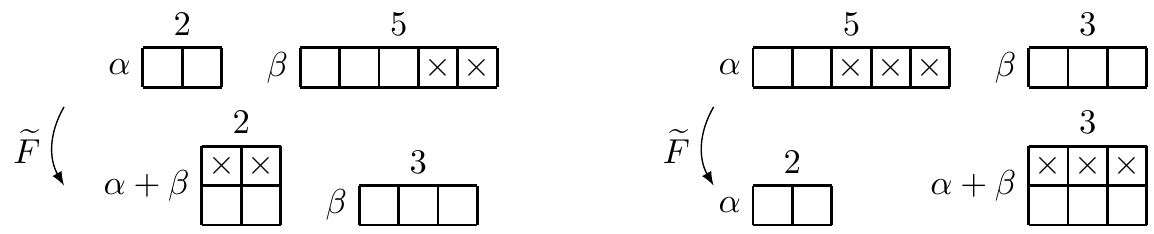}
\end{center}
We thus constructed the following function 
\[
\widetilde F:(x,y,\alpha,\beta)\mapsto
\begin{cases}
(x,y-x,\alpha+\beta,\beta)    & \mbox{if } x<y, \\
(x-y,y,\alpha,\alpha+\beta)   & \mbox{if } x>y.
\end{cases}
\]
which is an explicit model of the \emph{natural extension} of $F$.
It is a bijection on the domain $D=\{(x,y,\alpha,\beta)\in\R^2\times\R^2\mid
x>0,y>0,\alpha>0,\beta>0\}$; indeed, if we define $\Lambda_1=\{(x,y)\in\R^2\mid 0<x<y\}$ and $\Lambda_2=\{(x,y)\in\R^2\mid 0<y<x\}$, we have $\widetilde F(\Lambda_1\times \Lambda)=\Lambda\times \Lambda_2$ and  $\widetilde F(\Lambda_2\times \Lambda)=\Lambda\times \Lambda_1$;  Since $\Lambda_1, \Lambda_2$ is a measurable partition of $\Lambda$, $\widetilde F$ sends a partition of $D$ to another.

Both branches of $\widetilde F$ can be written linearly as the action by a $4\times 4$
matrix of the form
\[
\left(\begin{array}{cc}
M^{-1} & 0 \\
0      & M^\top
\end{array}\right)
\quad
\text{where}\quad
M=
\left(\begin{array}{cc}
1  & 0 \\
1 & 1
\end{array}\right)
\quad	
\text{if } x<y
\quad
\text{and}\quad
M=
\left(\begin{array}{cc}
1 & 1 \\
0 & 1
\end{array}\right)
\quad	
\text{if } x>y.
\]
Hence $\widetilde F$ preserves Lebesgue measure, since its jacobian is $1$.  It also preserves the form $x\alpha+y\beta$, that is, the total  area of the
rectangles. 

We want to find a natural extension and an invariant measure for the projective map. For this, we introduce the flow $\varphi_t$:
\[
\varphi_t:(x,y,\alpha,\beta)\mapsto
(xe^t,ye^t,\alpha e^{-t},\beta e^{-t}).
\]
This flow has jacobian $1$, thus preserves Lebesgue measure, and also preserves the form $x\alpha+y\beta$. It commutes with $\widetilde F$, that is, $\varphi_t(\widetilde F(\bx))=\widetilde F(\varphi_t(\bx))$.

Define now a subset $D_1$ of codimension 1 of $D$ by
$D_1=\{(x,y,\alpha,\beta)\in D \mid x\alpha+y\beta=1\}$, and a surface
$\Sigma\subseteq D_1$ by $\Sigma=\{(x,y,\alpha,\beta)\in D_1 \mid x+y=1\}$. 

$D_1$ is invariant by $\widetilde F$ and $\varphi_t$; $\widetilde F$ acts without periodic points since it decreases strictly $x+y$. The surface $\Sigma$ is a section for $\varphi_t$, since, for any $(x,y,\alpha,\beta)\in D_1$, there exists a unique $\tau=-\log(x+y)$ such that $\varphi_{\tau}(x,y,\alpha,\beta)\in\Sigma$.

Since $\varphi_t$ commutes with $\widetilde F$, we can consider the induced
flow on $D_1/\widetilde F$ and the first return map of
this flow to the projection $\Sigmab$ of the section $\Sigma$ on
$D_1/\widetilde F$. By construction, the effect of this map on the first
coordinates is the effect of the map $F$ followed by a renormalization, that
is, the projective map $f$.

We consider the change of coordinates $(x,y,\alpha,\beta)\mapsto
(x, \alpha-\beta, -\log(x+y),x\alpha+y\beta)=(x,\alpha',\tau, e)$; computation
shows that the jacobian is 1, so Lebesgue measure is given in these coordinates as $dx\, d\alpha'\, d\tau\, de$. The domain $D_1$ is given by $e=1$, and the surface of section $\Sigma$ by $\tau=0$, hence the first return map $\widetilde f$ must leave invariant the measure $dx\, d\alpha'$. Indeed, a straightforward computation shows that $\widetilde f$ is given in
the coordinates of $\Sigma$ by
\[
\begin{array}{rrcl}
    \widetilde f: & \Sigma  & \to     & \Sigma\\
   & (x,\alpha')     & \mapsto & 
\begin{cases}
\left(\frac{x}{1-x} , 1-x+\alpha'(1-x)^2\right)  & \mbox{if } x<\frac{1}{2}, \\
\left(2-\frac{1}{x},\alpha'x^2-x\right)   & \mbox{if } x>\frac{1}{2}.
\end{cases}
\end{array}
\]
and this map has jacobian 1.

We can describe more precisely the surface $\Sigma$ by (see Figure~\ref{fig:FareyNatExt}):
\begin{align*}
    \Sigma  
&= \{(x,\alpha',\tau,e)\mid \tau=0, e=1, \alpha'+\beta>0,\beta>0, 
x\alpha' + \beta=1\}\\
&= \{(x,\alpha',0,1)\mid \alpha'+1-x\alpha'>0,1-x\alpha'>0\}\\
&= \{(x,\alpha',0,1)\mid \frac{1}{x-1}<\alpha'<\frac{1}{x}\}
\end{align*}
Therefore the invariant density of $f$ is
\[
\delta(x)=
\int
_{\frac{1}{x-1}}
^{\frac{1}{x}}
1\, d\alpha'
=
\frac{1}{x}
- \frac{1}{x-1}
=\frac{1}{x(1-x)}.
\]
\begin{figure}[h]
\begin{center}
\includegraphics{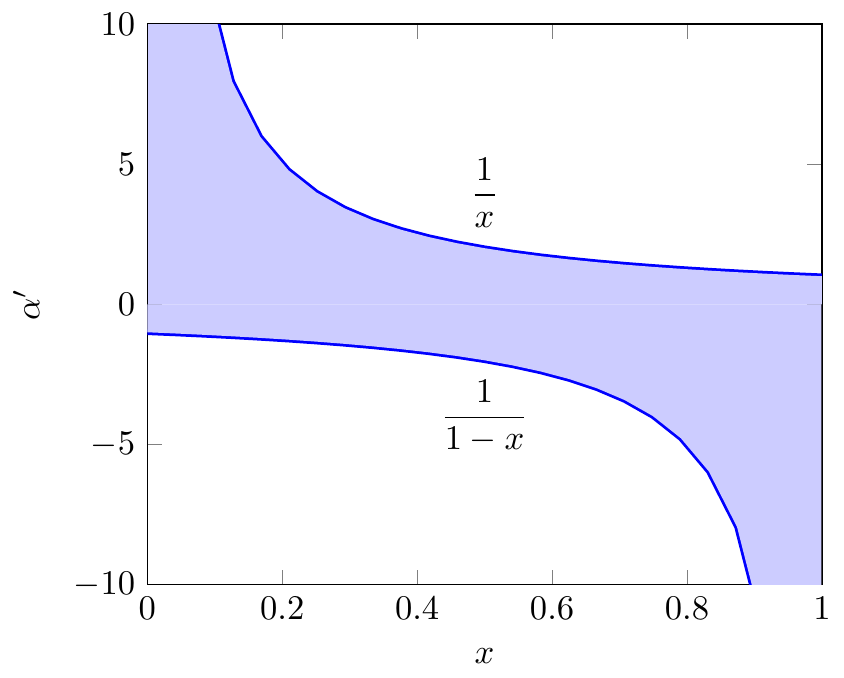}
\end{center}
\caption{
The domain of the natural extension of the unsorted Farey map.
}
\label{fig:FareyNatExt}
\end{figure}

One easily check that this density is invariant; a straightforward computation shows that, if $y_1=\frac {x}{1+x}$ and $y_2=\frac 1{2-x}$ are the preimages of the point $x$, we have $\delta(x)=\frac{\delta(y_1)}{f'(y_1)}+\frac{\delta(y_2)}{f'(y_2)}$.

\begin{remark}
That density is unbounded; moreover, the total mass of the measure is infinite. The reason is that we have two indifferent fixed points at both extremities of the interval, because we have used the additive algorithm. We could  accelerate the map around the fixed points by using the multiplicative algorithm, and recover the classical finite Gauss measure (more exactly, a variant of it, since we use the unsorted map).
\end{remark}

\begin{remark}
There is one more structure: the map $\widetilde F$ and the flow $\varphi_t$ leave invariant the symplectic form $dx\wedge d\alpha+dy\wedge d\beta$, and the flow $\varphi_t$ is the hamiltonian flow associated with the hamiltonian $x\alpha+y\beta$. In fact, the set $D/\widetilde F$ can be identified to  the tangent bundle of the modular surface, given in coordinates, and $\varphi_t$ is the geodesic flow. 
\end{remark}

\section{Constructing the measure from the natural
extension}\label{sec:constructing}

One can extend the method of the previous section to a large range of continued fraction algorithm. We recall the method proposed by Arnoux and Nogueira \cite{MR1251147} to
find heuristics for a geometric model of the natural extension of a multidimensional continued fraction algorithm, and an explicit formula for the invariant density (Gauss measure) in some cases.

Let $\Lambda$ be a subcone of the positive cone $\R^d_+$  (Two interesting cases are when $\Lambda=\R^d_+$ and when $\Lambda=\{(x_1,x_2,\ldots,x_d)\in\R^d\mid 0<x_1<x_2<\ldots<x_d\}$
consists of sorted entries). A \emph{Multidimensional Continued Fraction (MCF) algorithm} is a function
\[
\begin{array}{rrcl}
F: & \Lambda & \to     & \Lambda\\
   & \bx     & \mapsto & M(\bx)^{-1}\cdot\bx.
\end{array}
\]
where $M$ is an homogeneous  function of degree $0$,  piecewise constant on subcones,  that associates to each $\bx\in\Lambda$ an invertible matrix $M(\bx)$.  Classical references for MCF algorithms are \cite{schweiger,BRENTJES}. In the cases we consider, the entries of $M(\bx)$ are nonnegative integers.

To the function $F$, we can associate its projective version $f$, defined in a canonical way from the piecewise linear map on the projective space $P(\Lambda)$. For explicit computation, one should take coordinates, and a nice way to do that is to consider the codimension 1 compact domain  in $\Lambda$, $\Delta = \{ \bx\in\Lambda \mid \Vert\bx\Vert=1\}$ for some
norm $\Vert\cdot\Vert$; it is explicitly given by:
\[
\begin{array}{rrcl}
f: & \Delta  & \to     & \Delta\\
& \bx     & \mapsto & \frac{F(\bx)}{\Vert F(\bx)\Vert}.
\end{array}
\]

\begin{assumption}
We will always suppose that, out of a set of measure 0, we have $\|F(x)\|<\|x\|$; this is the case in all our examples.
\end{assumption}

Even if there is no canonical coordinates on $P(\Lambda)$, and hence no canonical Lebesgue measure, there is a natural Lebesgue measure class, since all changes of coordinates are piecewise projective. In many cases, there is a unique invariant (up to a constant) measure in this class, the Gauss measure, given by an invariant density. We are interested by explicit formulas for this invariant density, and for this, we try to build a geometric model for the natural extension of $f$ with good properties. 

A potential geometric model for the \emph{natural extension} $\widetilde F$ of the piecewise linear function $F$ can be defined as:
\[
\begin{array}{rccl}
\widetilde F: & \Lambda\times\R_+^{d}   & \to & \Lambda\times\R_+^{d} \\
	  &
\left(\begin{array}{r}
\bx \\
\ba
\end{array}\right)
& \mapsto   &
\left(\begin{array}{cc}
M(\bx)^{-1} & 0 \\
0          & M(\bx)^\top
\end{array}\right)
\left(\begin{array}{r}
\bx \\
\ba
\end{array}\right)
\end{array}
\]
To normalize the vector $\bx$ after the application of $\widetilde F$, we  define  the flow $\varphi_t$ as
\[
\begin{array}{rccl}
\varphi_t: & \Lambda\times\R_+^{d}   & \to & \Lambda\times\R_+^{d} \\
    &
\left(\begin{array}{r}
\bx \\
\ba
\end{array}\right)
& \mapsto   &
\left(\begin{array}{cc}
e^t & 0 \\
    0          & e^{-t}
\end{array}\right)
\left(\begin{array}{r}
\bx \\
\ba
\end{array}\right)
\end{array}
\]
for all $t\in\R$.
Note that the following properties are verified:
\begin{itemize}
\item $\varphi_t$ and $\widetilde F$ are well defined, since the matrix $M(x)$ is nonnegative and the definition set is a product of cones;
\item $\varphi_t$ and $\widetilde F$ both have jacobian $1$;
\item $\varphi_t$ and $\widetilde F$ both preserve the scalar product $\langle\bx,\ba\rangle$;
\item $\varphi_t$ commutes with $\widetilde F$: $\varphi_t\circ\widetilde
F = \widetilde F\circ\varphi_t$.
\end{itemize}

\begin{remark} The map $\widetilde F$ has no reason to be a bijection, and the main problem is to find a conical domain $D\subset \Lambda\times\R_+^{d} $ such that $\widetilde F$  is a bijection on $D$. 
\end{remark}

{\em From now on, we suppose that we have found such a domain $D$ of positive measure.} Then  $\varphi_t$ and $\widetilde F$ preserve Lebesgue measure on $D$, since they are bijections with jacobian 1. But this measure is not  very useful, because not only it has infinite mass, but its projection to $\Lambda$ has infinite density. We want to restrict the domain.

\begin{definition}
We define $D_1= \{(\bx,\ba)\in D|\langle\bx,\ba\rangle=1\}$ 
\end{definition}

We can define a natural Lebesgue measure on $D_1$ by considering coordinates on $D_1$ complemented by $\langle\bx,\ba\rangle$ to get a global coordinate system on $D$, and disintegrating Lebesgue measure. The functions $\widetilde F$ and $\varphi_t$ are well-defined on $D_1$ and both preserve locally Lebesgue measure on $D_1$. Since they are bijective, they also preserve it globally. 

\begin{figure}[h]
\begin{center}
    \includegraphics{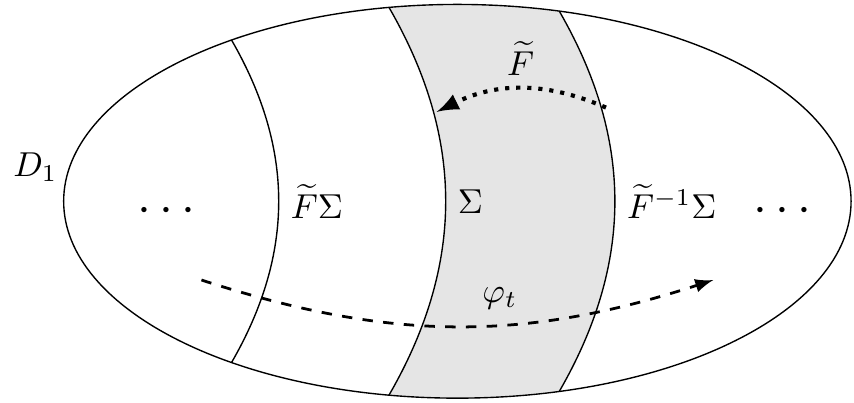}
\end{center}
\caption{The function $\widetilde F$ applied iteratively on the section
$\Sigma\subset D_1$. The gray region corresponds to a fundamental
domain of the action of $\widetilde F$ on $D_1$ and can be identified to $\Omega$.
}
\label{fig:patate}
\end{figure}

\begin{definition}
We define $\Omega=D_1/\widetilde F$, the set of orbits of the map $\widetilde F$.
\end{definition}

Since $\varphi_t$ and $\widetilde F$ commute, we can make $\varphi_t$ act on the quotient space $\Omega$. We can assimilate $\Omega$ to a fundamental domain of
the action of $\widetilde F$ on $D_1$ (this is well defined, since, except for the fixed points of $F$ which have measure 0,  $\widetilde F$ decreases the norm of $\bf x$).

\begin{definition}
We define $\Sigma= \{(\bx,\ba)\in D_1 \mid \Vert\bx\Vert=1\}$, and denote by $\Sigmab$ its projection to $\Omega$.
\end{definition}

\begin{remark}
There is a canonical projection $\pi: D_1\to D_1/\widetilde F$, mapping
each point on its orbit under $\widetilde F$. 
The restriction $\pi: 
\Sigma\to \Sigmab$ is a bijection, up to a set of measure 0: it is onto since
we have $\Sigmab=\pi(\Sigma)$ by definition, and if it is not one-to-one, there are two distinct elements
$(\bx,\ba),(\by,\bb)\in \Sigma$ and an integer $n> 0$ such that $(\by,\bb)=\widetilde F^n (\bx,\ba)$; this implies that $\|\by\|=\|F^n(\bx)\|=\|\bx\|$, which can only happen on a set of measure 0 by assumption. The dynamics happen in the quotient set $\Omega$, and we are mainly interested in the first return map to $\Sigmab$, but
all explicit calculus must be done in  $D_1$ and $\Sigma$ instead of $\Omega$ and $\Sigmab$.
\end{remark}

\begin{remark}
The surface $\Sigma$ is a {\em section} for the flow $\varphi_t$: for any
$(\bx,\ba)\in D_1$, there is a unique $t_{\bx}=-\log(\Vert\bx\Vert)$, depending only
on $\bx$,  such that $\varphi_{t_{\bx}}(\bx,\ba)\in \Sigma$.
\end{remark}

We can define a fundamental domain for the action of $\widetilde F$ as
\[
    \{ \varphi_t(\bx,\ba) \mid (\bx,\ba)\in\Sigma, 0\leq t < t_{F(\bx)}\}
\]
(see Figure~\ref{fig:patate}).
We consider the first return map   $\widetilde f:\Sigmab\to\Sigmab$ of the flow on $\Omega$. This map can
be seen in the fundamental domain as the composition of $\widetilde F$ with
the flow: $\widetilde f(\bx,\ba)= \varphi_{t_{F(\bx)}}\circ \widetilde F(\bx,\ba)$.

Since $\widetilde F$ preserves Lebesgue measure on $D_1$, there is a well-defined Lebesgue measure on $\Omega$, which is preserved by $\varphi_t$; since $\Sigmab$ is a section, we can define a transverse invariant measure $\mu$ on $\Sigmab$ in the usual way: if $m$ is Lebesgue measure on $\Omega$ and $U\subset\Sigmab$ is a measurable set, we define 
$$\mu(U)=\lim_{t\to 0}\frac{m\{\varphi_s(u)|u\in U, 0\le s\le t\}}{t}$$

From the definition, we have $\widetilde f(\bx,\ba) = (f(\bx), M(\bx)^\top\ba\cdot\Vert F(\bx)\Vert)$,  and $\widetilde f$ can be projected on the application $f$.  Hence we can find an invariant measure  for $f$ equivalent to Lebesgue measure  by partial integration on the section $\Sigmab$.

\begin{proposition} \label{prop:arnouxnogueira}{\rm\cite{MR1251147}}
If the measure of $\Sigma$ is positive, by choosing a coordinate system $(\bx, \ba)$ in which the invariant measure  is written $d\bx\, d\ba$, the invariant measure of $f$ can
be computed by the formula:
    \[
	\displaystyle
	\delta(\bx) = \int_{\{\ba: (\bx,\ba)\in\Sigma\}} 1 \,d\ba
    \]
\end{proposition}

We now describe an explicit change of variables that is useful to compute explicitly all these measures in the case of the norm $\|\bx\|=\sum_{i=1}^d |x_i|$. Define $y_i=x_i$ for $i<d$, and $\tau=-\log(\sum_{i=1}^d |x_i|)$; define $b_i=a_i-a_d$ for $i<d$ and $e=\sum_{i=1}^d a_i x_i$. A straightforward computation shows that the change of coordinates $(\bx, \ba)\mapsto (\by, \bb, \tau, e)$ has jacobian 1. The domain $D_1$ is defined by $e=1$, so the invariant measure on $\Omega$ is given by $\prod_{i=1}^{d-1}dy_i\prod_{i=1}^{d-1} db_i\,d\tau$; the section $\Sigma$ is defined by $\tau=1$; since the coordinate $\tau$ is, up to a sign, the time coordinate of the flow, the invariant transverse measure for the flow, that is, the invariant measure for $\widetilde f$, is given explicitly by $\prod_{i=1}^{d-1}dy_i\prod_{i=1}^{d-1} db_i$.

\begin{remark} We have made arbitrary choices. The first one was the choice of a coordinates for the projective space $P(\Lambda)$; we could choose another section, for example $x_d=1$ (we could also take another norm, but it seems in general more convenient to choose an affine section). The second one, which is related, is the choice of the section $\Sigma$. The flow $\varphi_t$ and the space $\Omega$ are more intrinsic; different choices lead to different forms of the same continued fraction; this explains why we can find a large range of formulas for apparently similar continued fractions, based on a different choice of coordinates.
\end{remark}

\section{ The Reverse algorithm}\label{sec:reverse}

The Arnoux-Rauzy algorithm is a partial algorithm, defined on the positive cone by subtracting the two smallest coordinates from the largest one. It is only defined if the largest coordinate is larger than the sum of the two smaller coordinates, that is, on the complement of the set of coordinates  which satisfy the triangular inequality. The set of points with infinite orbits, known as the Rauzy gasket \cite{2013_arnoux_gasket}, has measure 0, which make it unusable to find rational approximations of most points in the positive cone.

It can be completed in various ways by defining it on this missing set; for example, one can consider the simplest map which sends the central cone onto the positive cone, $(x,y,z)\mapsto (-x+y+z, x-y+z,x+y-z)$.

More formally, we can use the notations of the previous section, and define Reverse algorithm on the following partition of
$\Lambda = \{ (x,y,z) \in \R^3 \mid x>0, y>0, z>0\}$ up to a set of measure
zero:
\begin{align*}
\Lambda_1 &= \{(x,y,z) \in\Lambda \mid   x>y+z      \}, \\
\Lambda_2 &= \{(x,y,z) \in\Lambda \mid   y>x+z      \}, \\
\Lambda_3 &= \{(x,y,z) \in\Lambda \mid   z>x+y      \}, \\
\Lambda_4 &= \{(x,y,z) \in\Lambda \mid x<y+z, y<x+z, z<x+y \}
\end{align*}

The name {\em Reverse} comes from the fact that its effect, as a projective map acting on the central triangle $\Lambda_4$ of the unit simplex, is an homothety of ratio 2, followed by a central symmetry which reverses the vector with respect to the center of the simplex.

\begin{figure}[h]
\begin{center}
    \includegraphics[height=6cm]{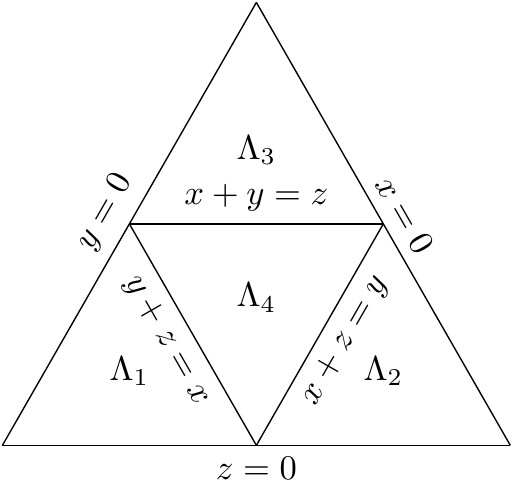}
    \includegraphics[height=6cm]{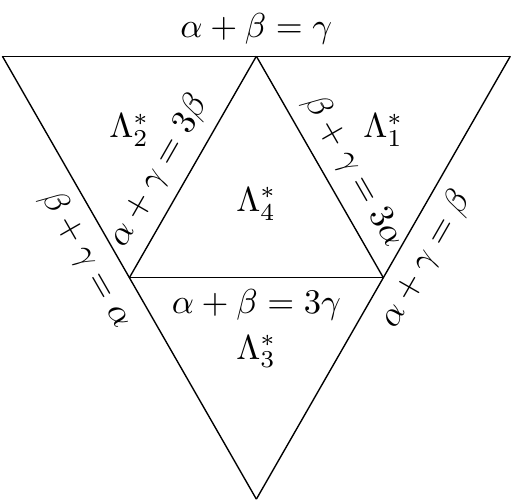}
\end{center}
\caption{Left: the trace on the unit simplex of the partition of $\Lambda$ into $\Lambda_i$ for
    $i\in\{1,2,3,4\}$.
Right: the partition of $\Lambda^*$ into $\Lambda^*_i$ for $i\in\{1,2,3,4\}$.}
\label{fig:lambdailambdaistar}
\end{figure}

We define the four matrices:
\[
    A_1 =
\left(\begin{array}{rrr}
1 & 1 & 1 \\
0 & 1 & 0 \\
0 & 0 & 1
\end{array}\right),
    A_2 =
\left(\begin{array}{rrr}
1 & 0 & 0 \\
1 & 1 & 1 \\
0 & 0 & 1
\end{array}\right),
    A_3 =
\left(\begin{array}{rrr}
1 & 0 & 0 \\
0 & 1 & 0 \\
1 & 1 & 1
\end{array}\right),
    A_4 =
\left(\begin{array}{rrr}
0 & \frac{1}{2} & \frac{1}{2} \\
\frac{1}{2} & 0 & \frac{1}{2} \\
\frac{1}{2} & \frac{1}{2} & 0
\end{array}\right).
\]
and the matrix function $M:\Lambda\to GL(3,\Q)$  such that $M (\bx)=  A_i$  if
and only $\bx\in\Lambda_i$. Recall that this matrix function defines the functions 
$F(\bx)=M(\bx)^{-1}\bx$, 
$f(\bx)=F(\bx)/\Vert F(\bx)\Vert_1$ and 
$\widetilde F(\bx,\ba)=(M(\bx)^{-1}\bx, M(\bx)^\top\ba)$. 

We show in the left of figure \ref{fig:lambdailambdaistar} the trace on the unit simplex of the partition of $\Lambda$; each branch of the map $F$ sends $\Lambda_i$ to all of $\Lambda$.

Explicitly, the function $\widetilde F$ is given by 
\[
\widetilde F:   (x,y,z,\alpha,\beta,\gamma)
 \mapsto   
\begin{cases}
(x-y-z,y,z,\alpha,\alpha+\beta,\alpha+\gamma)   & \mbox{if } (x,y,z)\in\Lambda_1,\\
(x,y-x-z,z,\alpha+\beta,\beta,\beta+\gamma)   & \mbox{if } (x,y,z)\in\Lambda_2,\\
(x,y,z-x-y,\alpha+\gamma,\beta+\gamma,\gamma)   & \mbox{if } (x,y,z)\in\Lambda_3,\\
(-x+y+z,x-y+z,x+y-z,\frac{\beta+\gamma}{2},\frac{\alpha+\gamma}{2},\frac{\alpha
+\beta}{2}) & \mbox{if } (x,y,z)\in\Lambda_4.
\end{cases}
\]

One can easily find a cone $\Lambda^*$ such that $\widetilde F$ is a bijection on $D= \Lambda\times\Lambda^*$: numerical experimentations show that after few iterations of  $\widetilde F$, $(\alpha,\beta, \gamma)$  belongs to the subset of $\Lambda_4\subset \Lambda$ of triples which satisfy the triangular inequality. Indeed, for any $(\alpha,\beta, \gamma)\in \Lambda$,  $(\frac{\beta+\gamma}{2},\frac{\alpha+\gamma}{2},\frac{\alpha+\beta}{2})$ satisfies the triangular inequality; and if  $(\alpha,\beta, \gamma)\in \Lambda_4$, then  $(\alpha,\alpha+\beta, \alpha+\gamma)$ also satisfies the triangular inequality. Since 
almost all orbits of $F$ cross $\Lambda_4$, with probability 1, after a finite time, the second part of the coordinates of $\widetilde F^n(\bx,\ba)$ enters $\Lambda_4$ and never escapes. 

Hence we define $\Lambda^* = \Lambda_4$, and a partition of $\Lambda^*$ by 
\begin{align*}
\Lambda^*_1 &= \{(\alpha,\beta,\gamma)\in \Lambda^* \mid \alpha<\frac{\alpha+\beta+\gamma}4 \},\\
\Lambda^*_2 &= \{(\alpha,\beta,\gamma)\in \Lambda^* \mid \beta<\frac{\alpha+\beta+\gamma}4 \},\\
\Lambda^*_3 &= \{(\alpha,\beta,\gamma)\in \Lambda^* \mid \gamma<\frac{\alpha+\beta+\gamma}4 \},\\
\Lambda^*_4 &= \{(\alpha,\beta,\gamma)\in \Lambda^*
\mid
\alpha,\beta, \gamma >\frac{\alpha+\beta+\gamma}4\}.
\end{align*}

\begin{lemma}
$\{\Lambda^*_1, \Lambda^*_2, \Lambda^*_3, \Lambda^*_4\}$ is a partition of $\Lambda^*$, up to a set of measure 0.
\end{lemma}
\begin{proof}
It is clear from the definition that $\Lambda^*_4$ is disjoint from the other 3, since it is defined by opposite inequalities. Consider a point in the intersection of $\Lambda^*_1$ and $\Lambda^*_2$; we have $\alpha, \beta<\frac{\alpha+\beta+\gamma}4$, hence $\gamma>\frac{\alpha+\beta+\gamma}2$, which is incompatible with the triangular inequality. The same proof is valid for the two other intersections, hence these four sets are disjoint.

The equality $ \alpha=\frac{\alpha+\beta+\gamma}4$ defines a set of
codimension 1 and measure 0, and similarly for the two other ones. For any
other point in $\Lambda^*$, either all of $\alpha,\beta, \gamma$ are larger
than $\frac{\alpha+\beta+\gamma}4$, which defines $\lambda^*_4$, or one (and
only one, by the previous paragraph) is  smaller,  which defines one of the
other simplices; hence they form a partition, up to a set of measure 0.  We
show in the right of figure \ref{fig:lambdailambdaistar} the trace on the unit simplex of this partition.
\end{proof}

As an immediate consequence, we obtain a partition of $D$: 

\begin{corollary} Up to a set of measure 0, we can write the set $D=\Lambda\times\Lambda^*$ as a disjoint union 
$$D=\coprod_{i=1}^4 \Lambda_i\times\Lambda^*=\coprod_{i=1}^4\Lambda\times\Lambda^*_i $$
\end{corollary}

Hence, to prove that $\widetilde F$ is a bijection on $D$, since each branch of $\widetilde F$ is a non-degenerate linear map, it suffices to prove:

\begin{lemma}\label{lem:subsetT}
 We have
     $\widetilde F(\Lambda_i\times\Lambda^*)= \Lambda\times\Lambda^*_i$.
\end{lemma}

\begin{proof}
We check directly by computation that $F(\Lambda_i)=\Lambda$. We want to show that if $\bx\in\Lambda_i$, then $A_i^\top\ba\in\Lambda^*_i$.
We use the notation
$(x',y',z',\alpha',\beta',\gamma')=\widetilde F(x,y,z,\alpha,\beta,\gamma)$
and we proceed case by case.

If $\bx\in \Lambda_1$,
then
$ \alpha'+\beta'+\gamma' = 3\alpha+\beta+\gamma > 4\alpha = 4\alpha'$.
Therefore $\widetilde F(x,y,z,\alpha,\beta,\gamma)\in \Lambda\times\Lambda^*_1$.
A similar proof applies for $\Lambda_2$ and $\Lambda_3$.

If $\bx\in \Lambda_4$, then $\alpha'+\beta' +\gamma'=\alpha+\beta+\gamma< 2\alpha+2\beta= 4\gamma'$; a similar proof applies to $\alpha'$ and $\beta'$, which proves that  $ \widetilde F(x,y,z,\alpha,\beta,\gamma)\in \Lambda\times\Lambda^*_4$.

We have proved inclusion; to prove equality, it is enough to show that the matrix $A_i^\top$sends the extreme points of $\Lambda^*$ to the extreme points of $\Lambda^*_i$, which is done by computation.
See Figure~\ref{fig:arnouxrevert}.
\end{proof}

\begin{figure}[h]
\begin{center}
\includegraphics{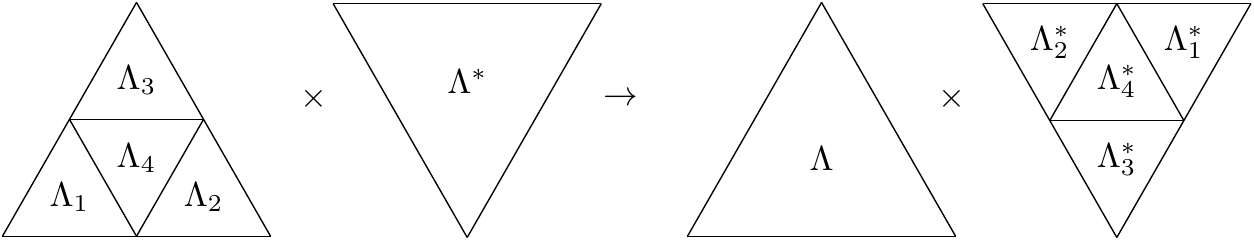}
\end{center}
\caption{
The function $\widetilde F:\Lambda_i\times\Lambda^*\to \Lambda\times\Lambda^*_i$ is
one-to-one and onto
for each $i\in\{1,2,3,4\}$.
}
\label{fig:arnouxrevert}
\end{figure}

We now apply the results of the previous section. Recall that
\[
D_1 = \{(x,y,z,\alpha,\beta,\gamma)\in D
\mid
x\alpha+y\beta+z\gamma=1
\}.
\]
and the surface of section is defined by 
\[
\Sigma=
\{(x,y,z,\alpha,\beta,\gamma)\in\R_+^6\mid
x\alpha+y\beta+z\gamma=1,
x+y+z=1,
\alpha+\beta>\gamma,
\alpha+\gamma>\beta,
\beta+\gamma>\alpha\}.
\]

It is convenient to take the coordinates on $D_1$ defined in the previous section:  
We keep variables $x$ and $y$ and change the coordinates 
$\alpha'=\alpha-\gamma$,
$\beta'=\beta-\gamma$,
$e=x\alpha+y\beta+z\gamma$,
$\tau=-\log(x+y+z)$; it is readily checked that the jacobian is one. 
Since the domain $D_1$ is defined by $e=0$, and coordinate $\tau$ is the return time of the flow to $\Sigmab$, we get
\[
\Sigma  
= \{(x,y,\alpha',\beta',\tau,e)\in\R_+^6\mid
\tau=0,
e=1,
\alpha'+\beta'+\gamma>0,
\alpha'+\gamma>\beta',
\beta'+\gamma>\alpha'\}
\]

Furthermore, the invariant measure for the return map $\widetilde f$ of the flow to $\Sigmab $ is $dx\, dy\, d\alpha'\, d\beta'$. From this, we obtain:

\begin{proposition} \label{prop:invariantmeasureARR}
The density function of the invariant measure of $f:\Delta\to\Delta$ for
the Reverse algorithm is
\[
\frac{1}{(1-x)(1-y)(1-z)}.
\]
\end{proposition}



\begin{proof} 
To obtain this density and according to
Proposition~\ref{prop:arnouxnogueira}, it suffices to integrate the measure
$d\alpha'\, d\beta'$ for a fixed $x,y$.
Given $(x,y)\in\R^2$, the set of admissible 
$(\alpha',\beta')\in\R^2$ satisfies
\[
\left\{
\begin{array}{l}
x\alpha'+y\beta'+\gamma=1,\\
\alpha'+\beta'+\gamma>0,\\
\alpha'+\gamma>\beta',\\
\beta'+\gamma>\alpha',
\end{array}
\right.
\quad
\text{or equivalently}
\quad
\left\{
\begin{array}{l}
(1-x)\alpha'+(1-y)\beta'+1>0,\\
(1-x)\alpha'+1>(1+y)\beta',\\
(1-y)\beta'+1>(1+x)\alpha'.
\end{array}
\right.
\]
This domain is the interior of a triangle of vertices
$\left(\frac{1}{x-1},0\right)$,
$\left(0,\frac{1}{y-1}\right)$,
$\left(\frac{1}{x+y},\frac{1}{x+y}\right)$
shown below.
\begin{center}
\includegraphics{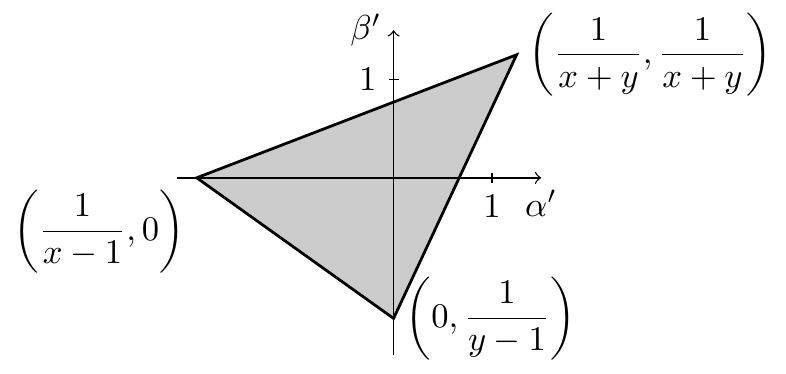}
\end{center}

\begin{lemma}\label{lem:triangle}
Let $a,b,c\in\R$.
The area of a triangle of vertices $(a,0)$, $(0,b)$ and $(c,c)$ is $\frac{1}{2}|ab-bc-ac|=\frac 12 |abc|\left|\frac 1c-\frac 1a-\frac 1b\right|$.
\end{lemma}

This lemma is proved by an easy determinant computation; it shows that the triangle we consider has area
$$\frac{1}{2}
\left(\frac{1}{(x+y)(x-1)(y-1)}\right)
\left\vert(x+y) - (x-1) - (y-1)\right\vert
=
\frac{1}{(x+y)(1-x)(1-y)}.$$

\qedhere
\end{proof}

\begin{remark} This density is not bounded: it tends to infinity at the three extreme points of the simplex. The reason is that these points are indifferent fixed points for the algorithm. This is a common feature for an additive continued fraction algorithm. However, and oppositely to what happens in dimension 1,  the total mass of that density is bounded; its value is 
\[
\int_0^1
\int_0^{1-x}
\frac{1}{(x+y)(1-x)(1-y)}
dydx
=
\int_0^1
\frac{2 \, \log\left(x\right)}{{\left(x + 1\right)} {\left(x - 1\right)}}
dx
=
\frac{\pi^2}{4}
\]
\end{remark}


\section{The Cassaigne algorithm}\label{sec:cassaigne}

This algorithm was suggested by Julien Cassaigne (DynA3S meeting, LIAFA, Paris,
October 12th,
2015)\footnote{\url{http://www.liafa.univ-paris-diderot.fr/dyna3s/Oct2015}} as a way to generate words of low factor complexity
($p(n)=2n+1$) in with arbitrary letter frequencies on a three-letter alphabet.
An interesting aspect of this algorithm is that there are always
only two branches.

More formally, we can use the notations of the previous section, and define
Cassaigne algorithm on the following partition of
$\Lambda = \{(x_1,x_2,x_3) \in \R^3 \mid x_1>0, x_2>0, x_3>0\}$
up to a set of measure zero:
\begin{align*}
\Lambda_a &= \{(x_1,x_2,x_3) \in\Lambda \mid   x_1>x_3     \}, \\
\Lambda_b &= \{(x_1,x_2,x_3) \in\Lambda \mid   x_3>x_1     \}.
\end{align*}
We define two matrices:
\[
    C_a =
\left(\begin{array}{rrr}
1 & 1 & 0 \\
0 & 0 & 1 \\
0 & 1 & 0
\end{array}\right)\qquad\text{and}\qquad
    C_b =
\left(\begin{array}{rrr}
0 & 1 & 0 \\
1 & 0 & 0 \\
0 & 1 & 1
\end{array}\right).
\]
We define the matrix function $M:\Lambda\to GL(3,\Q)$  such that $M (\bx)=  A_i$  if
and only $\bx\in\Lambda_i$ for $i\in\{a,b\}$. Recall that this matrix function
defines the functions $F(\bx)=M(\bx)^{-1}\bx$, 
$f(\bx)=F(\bx)/\Vert F(\bx)\Vert_1$ and 
$\widetilde F(\bx,\ba)=(M(\bx)^{-1}\bx, M(\bx)^\top\ba)$. 
We show in the left of Figure~\ref{fig:cassaignepartition} the trace on the
unit simplex of the partition of $\Lambda$; each branch of the map
$F$ sends $\Lambda_i$ to all of $\Lambda$ for $i\in\{a,b\}$.
\begin{figure}[h]
\begin{center}
    \includegraphics[height=6cm]{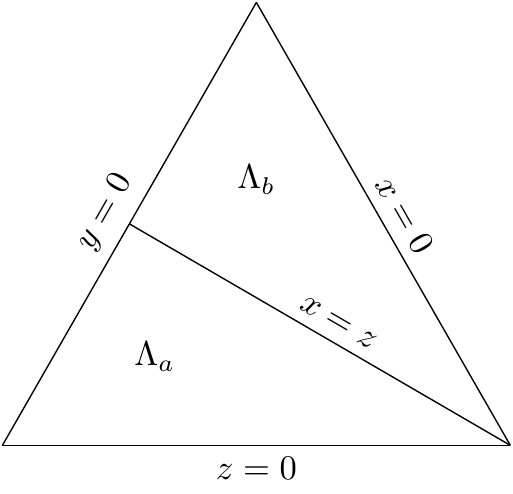}
    \includegraphics[height=6cm]{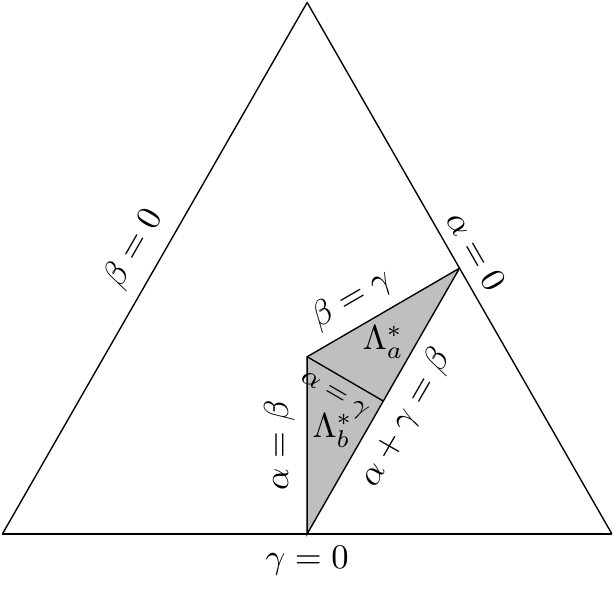}
\end{center}
\caption{Left: the trace on the unit simplex of the partition of $\Lambda$
    into $\Lambda_a$ and $\Lambda_b$.
Right: the partition of $\Lambda^*$ into $\Lambda^*_a$ and $\Lambda^*_b$.}
\label{fig:cassaignepartition}
\end{figure}

Explicitly, the function $\widetilde F$ is given by 
\[
\widetilde F:   (x,y,z,\alpha,\beta,\gamma)
 \mapsto   
\begin{cases}
(x-z,z,y,\alpha,\alpha+\gamma,\beta)   & \mbox{if } (x,y,z)\in\Lambda_1,\\
(y,x,z-x,\beta,\alpha+\gamma,\gamma)   & \mbox{if } (x,y,z)\in\Lambda_2.
\end{cases}
\]

One can easily find a cone $\Lambda^*$ such that $\widetilde F$ is a bijection
on $D= \Lambda\times\Lambda^*$: numerical experimentations show that after few
iterations of  $\widetilde F$, $(\alpha,\beta, \gamma)$  belongs to the subset
of $\Lambda^*\subset \Lambda$ 
\[
    \Lambda^* = \{(\alpha,\beta,\gamma)\in \R^3 \mid
    \max\{\alpha,\gamma\}<\beta<\alpha+\gamma\}.
\]
We define a partition (see the right of Figure~\ref{fig:cassaignepartition})
of $\Lambda^*$ by 
\begin{align*}
\Lambda^*_a &= \{(\alpha,\beta,\gamma)\in \Lambda^* \mid \alpha<\gamma \},\\
\Lambda^*_b &= \{(\alpha,\beta,\gamma)\in \Lambda^* \mid \alpha>\gamma \}.
\end{align*}

\begin{lemma}
$\{\Lambda^*_a, \Lambda^*_b\}$ is a partition of $\Lambda^*$, up to a set of measure 0.
\end{lemma}
\begin{proof}
    Trivial.
\end{proof}

As an immediate consequence, we obtain a partition of $D$: 

\begin{corollary} Up to a set of measure 0, we can write the set $D=\Lambda\times\Lambda^*$ as a disjoint union 
    $$D=\coprod_{i\in\{a,b\}} \Lambda_i\times\Lambda^*=
        \coprod_{i\in\{a,b\}} \Lambda\times\Lambda^*_i $$
\end{corollary}

Hence, to prove that $\widetilde F$ is a bijection on $D$, since each branch of $\widetilde F$ is a non-degenerate linear map, it suffices to prove:

\begin{lemma}\label{lem:subsetTcassaigne}
We have $\widetilde F(\Lambda_a\times\Lambda^*)= \Lambda\times\Lambda^*_a$
and     $\widetilde F(\Lambda_b\times\Lambda^*)= \Lambda\times\Lambda^*_b$.
\end{lemma}

\begin{proof}
We check directly by computation that $F(\Lambda_i)=\Lambda$. 
We want to show that if $\bx\in\Lambda_i$ and $\ba\in\Lambda^*$, 
then $C_i^\top\ba\in\Lambda^*_i$.  
It is enough to show that the
matrix $C_i^\top$ sends the extreme points of $\Lambda^*$ to the extreme points
of $\Lambda^*_i$ for $i\in\{a,b\}$. This is done by the following matrix computation 
\[
    C_a^\top
\left(\begin{array}{rrr}
1 & 1 & 0 \\
1 & 1 & 1 \\
0 & 1 & 1
\end{array}\right)
=
\left(\begin{array}{rrr}
1 & 1 & 0 \\
1 & 2 & 1 \\
1 & 1 & 1
\end{array}\right)
\qquad
\text{and}
\qquad
    C_b^\top
\left(\begin{array}{rrr}
1 & 1 & 0 \\
1 & 1 & 1 \\
0 & 1 & 1
\end{array}\right)
=
\left(\begin{array}{rrr}
1 & 1 & 1 \\
1 & 2 & 1 \\
0 & 1 & 1
\end{array}\right)
\]
and this proves the equalities.
See Figure~\ref{fig:cassaigne}.
\end{proof}

\begin{figure}[h]
\begin{center}
\includegraphics{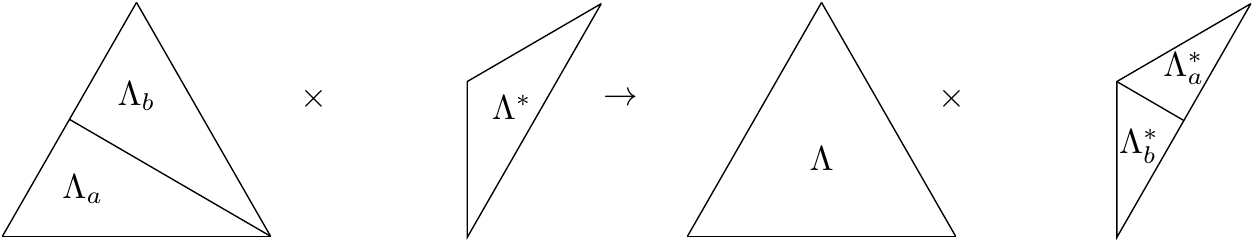}
\end{center}
\caption{
The function $\widetilde F:\Lambda_i\times\Lambda^*\to \Lambda\times\Lambda^*_i$ is
one-to-one and onto
for each $i\in\{1,2,3,4\}$.
}
\label{fig:cassaigne}
\end{figure}

We now apply the results of the previous section. Recall that
\[
D_1 = \{(x,y,z,\alpha,\beta,\gamma)\in D
\mid
x\alpha+y\beta+z\gamma=1
\}.
\]
and the surface of section is defined by 
\[
\Sigma= 
\{(x,y,z,\alpha,\beta,\gamma)\in\R_+^6\mid
x\alpha+y\beta+z\gamma=1,
x+y+z=1,
\alpha<\beta,
\gamma<\beta,
\beta<\alpha+\gamma\}.
\]
It is convenient to take the coordinates on $D_1$ defined in the previous section:  
We keep variables $x$ and $y$ and change the coordinates 
$\alpha'=\alpha-\gamma$,
$\beta'=\beta-\gamma$,
$e=x\alpha+y\beta+z\gamma$,
$\tau=-\log(x+y+z)$; it is readily checked that the jacobian is one. 
Since the domain $D_1$ is defined by $e=0$, and coordinate $\tau$ is the return time of the flow to $\Sigmab$, we get
\[
\Sigma
= \{(x,y,\alpha',\beta',\tau,e)\in\R_+^6\mid
\tau=0,
e=1,
\alpha'<\beta',
0<\beta',
\beta'<\alpha'+\gamma\}
\]
Furthermore, the invariant measure for the return map $\widetilde f$ of the
flow to $\Sigmab $ is $dx\, dy\, d\alpha'\, d\beta'$. From this, we obtain: 

\begin{proposition} \label{prop:invariantmeasureCassaigne}
The density function of the invariant measure of $f:\Delta\to\Delta$ for
the Cassaigne algorithm is
\[
\frac{1}{(1-x)(1-z)}.
\]
\end{proposition}

\begin{proof} 
To obtain this density and according to
Proposition~\ref{prop:arnouxnogueira}, it suffices to integrate the measure $
d\alpha'\, d\beta'$ for a fixed $x,y$. 
Given $(x,y)\in\R^2$, the set of admissible 
$(\alpha',\beta')\in\R^2$ satisfies
\[
\left\{
\begin{array}{l}
x\alpha'+y\beta'+\gamma=1,\\
\alpha'<\beta',\\
0<\beta',\\
\beta'<\alpha'+\gamma,
\end{array}
\right.
\quad
\text{or equivalently}
\quad
\left\{
\begin{array}{l}
\alpha'<\beta',\\
0<\beta',\\
(1-x)\alpha'+1>(1+y)\beta'.
\end{array}
\right.
\]
This domain is the interior of a triangle of vertices
$\left(\frac{1}{x-1},0\right)$,
$\left(0,0\right)$,
$\left(\frac{1}{x+y},\frac{1}{x+y}\right)$
shown below.
\begin{center}
\includegraphics{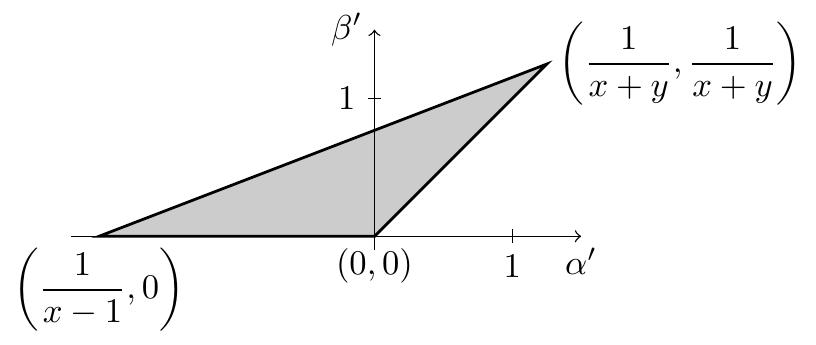}
\end{center}
The area of the triangle is
\[
\frac{1}{2}
\left(0-\frac{1}{x-1}\right)
\left(\frac{1}{x+y}\right)
=
\frac{1}{2(1-x)(x+y)}
\qedhere
\]
\end{proof}

\begin{remark} The total mass of that density is bounded; its value is 
\[
\int_0^1
\int_0^{1-x}
\frac{1}{2(1-x)(x+y)}
dydx
=
\int_0^1
\frac{\log\left(x\right)}{2 \, {\left(x - 1\right)}}
dx
=
\frac{\pi^2}{12}
\]
\end{remark}
\section{The Brun algorithm}\label{sec:brun}
The Brun algorithm is arguably the simplest multidimensional continued fraction algorithm; it is defined on the positive cone $\Lambda = \{(x_1,x_2,x_3) \in \R^3 \mid x_1>0, x_2>0, x_3>0\}$ by taking the second largest  coordinate from the largest one. Formally, we have 6 cases, and consider the following partition (up to a set of measure zero):
\[
\Lambda_\pi = \{(x_1,x_2,x_3) \in\Lambda \mid x_{\pi(1)}<x_{\pi(2)}<x_{\pi(3)}\},
\]
for $\pi\in\S_3$ where $S_3=\{123,132,231,213,312,321\}$ is the set of
permutations of $\{1,2,3\}$. We define the six elementary matrices:
\begin{align*}
    B_{123} =
\left(\begin{array}{rrr}
1 & 0 & 0 \\
0 & 1 & 0 \\
0 & 1 & 1
\end{array}\right),
    B_{132} =
\left(\begin{array}{rrr}
1 & 0 & 0 \\
0 & 1 & 1 \\
0 & 0 & 1
\end{array}\right),
    B_{231} =
\left(\begin{array}{rrr}
1 & 0 & 1 \\
0 & 1 & 0 \\
0 & 0 & 1
\end{array}\right),\\
    B_{213} =
\left(\begin{array}{rrr}
1 & 0 & 0 \\
0 & 1 & 0 \\
1 & 0 & 1
\end{array}\right),
    B_{312} =
\left(\begin{array}{rrr}
1 & 0 & 0 \\
1 & 1 & 0 \\
0 & 0 & 1
\end{array}\right),
    B_{321} =
\left(\begin{array}{rrr}
1 & 1 & 0 \\
0 & 1 & 0 \\
0 & 0 & 1
\end{array}\right).
\end{align*}
and the matrix function $M:\Lambda\to GL(3,\Q)$ is defined according to the
above partition:
\[
M: \bx  \mapsto  B_\pi \quad\text{ if and only if }\quad \bx\in\Lambda_\pi.
\]
As in Section~\ref{sec:constructing}, the matrix function $M(\bx)$ allows to
define the functions 
\[
F(\bx)=M(\bx)^{-1}\bx, \quad
f(\bx)=\frac{F(\bx)}{\Vert F(\bx)\Vert_1} \quad\text{and}\quad
\widetilde F(\bx,\ba)=(M(\bx)^{-1}\bx, M(\bx)^\top\ba). 
\]

The six branches of the map $F$ are not full; that is, their image is not the complete cone. It will be useful to define the subcones $\Theta_\pi$ by $$\Theta_\pi = \{(x_1,x_2,x_3) \in\Lambda \mid x_{\pi(1)}<x_{\pi(2)}\}.$$ 

\begin{remark}
Each $\Theta_{\pi}$ is a union  of three $\Lambda_{\sigma}$, for all permutation $\sigma$ such that $\sigma^{-1}\pi(1)<\sigma^{-1}\pi(2)$;  a direct computation shows that $F(\Lambda_\pi)=\Theta_{\pi}$, so that the $\Lambda_{\pi}$ form a Markov partition for $F$, which can be proved to be generating.
\end{remark}

We want to apply the techniques of Section~\ref{sec:constructing}; however, as experiments show, the domain for the natural extension can not be written as a global product, but as a disjoint union of products. Define for each permutation $\pi\in\S_3$,
\begin{align*}
    \Lambda^*_\pi &= \{(\alpha_1,\alpha_2,\alpha_3) \in\Lambda \mid
\alpha_{\pi(1)}<\alpha_{\pi(3)}<\alpha_{\pi(2)}\}
\\
\Theta^*_\pi &= \{(\alpha_1,\alpha_2,\alpha_3) \in\Lambda \mid \alpha_{\pi(1)}<\alpha_{\pi(3)}\}
\end{align*}

\begin{figure}[h]
\begin{center}
\includegraphics{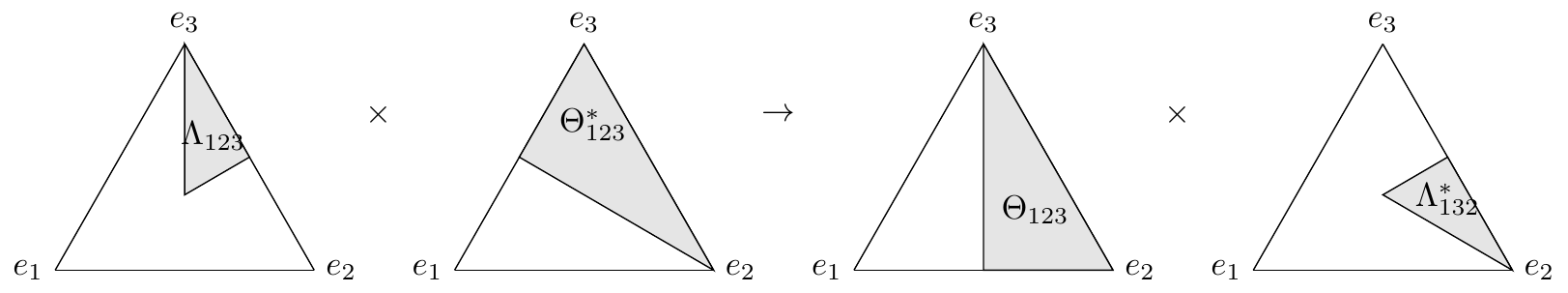}
\end{center}
\caption{
The function $\widetilde F:\Lambda_\pi\times \Theta^*_{\pi}\to
\Theta_\pi\times\Lambda^*_{\pi}$ is
one-to-one and onto
for each $\pi\in\S_3$.
}
\label{fig:brunpartitions}
\end{figure}

This is just a convenient notation; if we denote by $\rho=(123)$ the transposition exchanging 2 and 3, we have $\Lambda^*_{\pi}=\Lambda_{\pi\rho}$ and $\Theta^*_{\pi}=\Theta_{\pi\rho}$, and the previous remark applies: each $\Theta^*_{\pi}$ is the disjoint union of three sets $\Lambda^*_{\sigma}$.  We first prove a lemma.

\begin{lemma} We have: 
$$\bigcup_{\pi\in\S_3}\left(\Lambda_\pi\times \Theta^*_{\pi}\right)=\bigcup_{\pi\in\S_3}\left(\Theta_{\pi}\times\Lambda^*_{\pi}\right)$$
up to a set of measure 0, where both unions are disjoint.
\end{lemma}

\begin{proof} Since all $\Lambda_{\pi}$ are disjoint, the first union must be disjoint, and since all $\Lambda^*_{\pi}$ are disjoint, the second union is disjoint. 

By the previous remark, each $\Theta^*_{\pi} $ is the disjoint union of three sets $\Lambda_{\sigma}$, for those $\sigma\in S_3$  such that $\sigma^{-1}\pi(1)<\sigma^{-1}\pi(3)$. This equality is possible if, and only if,  $\sigma(3)=\pi(3)$, or $\sigma(3)=\pi(2)$ and  $\sigma(2)=\pi(3)$. Hence the first set is the union of all the 18 products $\Lambda_\sigma\times \Lambda_\pi$, for all pair $(\sigma,\pi)$ satisfying this condition.

The second set can be written $\bigcup_{\pi\in\S_3}\left(\Theta_{\pi}\times\Lambda^*_{\pi}\right)= \bigcup_{\pi\in\S_3}\left(\Theta_{\pi}\times\Lambda_{\pi\rho}\right)=\bigcup_{\pi\in\S_3}\left(\Theta_{\pi\rho}\times\Lambda_{\pi}\right)$, where $\rho=(132)$; and $\Lambda_{\sigma}\subset\Theta_{\pi\rho}$ if and only if $\sigma^{-1}\pi\rho(1)<\sigma^{-1}\pi\rho(2)$: by definition of $\rho$, this is the same condition, hence the two sets are equal.
\end{proof}

\begin{proposition}
The function $\widetilde F$ 
restricted to the domain $\Lambda_\pi\times \Theta^*_{\pi}$ is
one-to-one and onto 
$\Theta_\pi\times\Lambda^*_{\pi}$
for each permutation $\pi\in\S_3$.
The function $\widetilde F$ on the domain 
$D= \cup_{\pi\in\S_3}\left(\Lambda_\pi\times \Theta^*_{\pi}\right)$ 
is one-to-one and onto (See Figure~\ref{fig:brunpartitions}).
\end{proposition}

\begin{proof} We first prove that  $\widetilde F\left(\Lambda_\pi\times \Theta^*_{\pi}\right)=\Theta_\pi\times\Lambda^*_{\pi}$; by symmetry, it is enough to prove it for $\pi=(1,2,3)$. Let $(x_1, x_2, x_3, \alpha_1, \alpha_2, \alpha_3)\in \Lambda_\pi\times \Theta^*_{\pi}$, and let $(y_1, y_2, y_3, \beta_1, \beta_2, \beta_3)=\widetilde F(x_1, x_2, x_3, \alpha_1, \alpha_2, \alpha_3)=(x_1, x_2, x_3-x_2, \alpha_1, \alpha_2+\alpha_3, \alpha_3)$; we check that $y_1\le y_2 $ and $\beta_1\le \beta_3\le \beta_2$, which is the condition to be in $\Theta_\pi\times\Lambda^*_{\pi}$. To prove equality, it is enough to consider the reciprocal map $\widetilde F^{-1}(y_1, y_2, y_3, \beta_1, \beta_2, \beta_3)=(y_1, y_2, y_2+y_3, \beta_1, \beta_2-\beta_3, \beta_3)$, and to prove that $\widetilde F^{-1}\left(\Theta_\pi\times\Lambda^*_{\pi}\right) =\Lambda_\pi\times \Theta^*_{\pi}$, which is done exactly the same way.

Since each branch of $\widetilde F$ is an invertible linear map, it is a bijection between its domain and its image; by the previous lemma, $\widetilde F$ is a bijection on $D$.

\end{proof}

\begin{proposition} \label{prop:invariantmeasureBrun}
The density function of the invariant measure of $f:\Delta\to\Delta$ for
the Brun algorithm is
\[
\frac{1}{2\,x_{\pi(2)}(1-x_{\pi(2)})(1-x_{\pi(1)}-x_{\pi(2)})}
\]
on the part $\bx=(x_1,x_2,x_3)\in\Lambda_\pi\cap\Delta$.
\end{proposition}

\begin{proof}
    In this proof, we suppose that $\pi=(1,2,3)$.
    For Brun algorithm, the surface of section for the part
    $\vec{x}\in\Lambda_\pi$ is
\[
    \Sigma=
\{(x_1,x_2,x_3,\alpha_1,\alpha_2,\alpha_3)\in\R_+^{6} \mid \,
    \sum_{i=1}^{3}x_i\alpha_i=1,
    \sum_{i=1}^{3}x_i=1,
    \bx\in\Lambda_\pi,
    \ba\in\Theta^*_\pi\}.
\]
As explained in section \ref{sec:constructing},
we use an explicit change of variables that is useful to compute explicitly
all these measures in the case of the norm $\|\bx\|=\sum_{i=1}^3 |x_i|$.
Define $y_1=x_1$, $y_2=x_2$ and $\tau=-\log(\sum_{i=1}^3 |x_i|)$; define
$\beta_1=\alpha_1-\alpha_3$,
$\beta_2=\alpha_2-\alpha_3$ 
and $e=\sum_{i=1}^3 \alpha_i x_i$. A straightforward
computation shows that this change of coordinates has jacobian 1.
Below, we keep variables $x_1$ and $x_2$ instead of $y_1$ and $y_2$.
We get that $\Sigma$ is the set of 
$(x_1,x_2,\beta_1,\beta_2,\tau, e)\in\R^{6}$
    satisfying
\[
\left\{
\begin{array}{ll}
    \tau=0,&
\beta_1<0,\\
e=1,&
\beta_1+\alpha_3>0,\\
x_1<x_2,\qquad& 
\beta_2+\alpha_3>0,\\
&\alpha_3>0.
\end{array}
\right.
\]
The equalities $e=1$ and $\sum_{i=1}^3 x_i=1$ allows to substitute
$\alpha_3$ above. Indeed,
\[
    1 
    =e 
    = \sum_{i=1}^3 \alpha_i x_i
    = (\beta_1+\alpha_3) x_1 + (\beta_2+\alpha_3) x_2 + \alpha_3x_3
    = \beta_1x_1 +\beta_2x_2 + \alpha_3(x_1+x_2+x_3)
    = \beta_1x_1 +\beta_2x_2 + \alpha_3.
\]
So we get that $\Sigma$ is described by
\[
\left\{
\begin{array}{lll}
    \tau=0,&
\beta_1<0,\\
e=1,&
\beta_1x_1+\beta_2x_2-\beta_1<1,\\
x_1<x_2,\qquad&
\beta_1x_1+\beta_2x_2-\beta_2<1,\\
& \beta_1x_1+\beta_2x_2<1.
\end{array}
\right.
\]
From Proposition~\ref{prop:arnouxnogueira}, the density is equal to the
volume of the polytope $\{\ba: (\bx,\ba)\in\Sigma\}$:
\[
    \displaystyle
    \delta(x_1,x_2) = \int_{\{(\beta_1,\beta_2):
    (x_1,x_2,0,\beta_1,\beta_2,1)\in\Sigma\}} 1 \,d\beta_1\,d\beta_2
\]
Given $(x_1,x_2)\in\R^2$, the polytope corresponds to the set of
$(\beta_1,\beta_2)\in\R^2$ satisfying the four inequalities.
This domain is the interior of a triangle with vertices
$\left(0,\frac{1}{x_2-1}\right)$,
$\left(0,\frac{1}{x_2}\right)$,
$\left(\frac{1}{x_1+x_2-1},\frac{1}{x_1+x_2-1}\right)$
shown below.
\begin{center}
\includegraphics{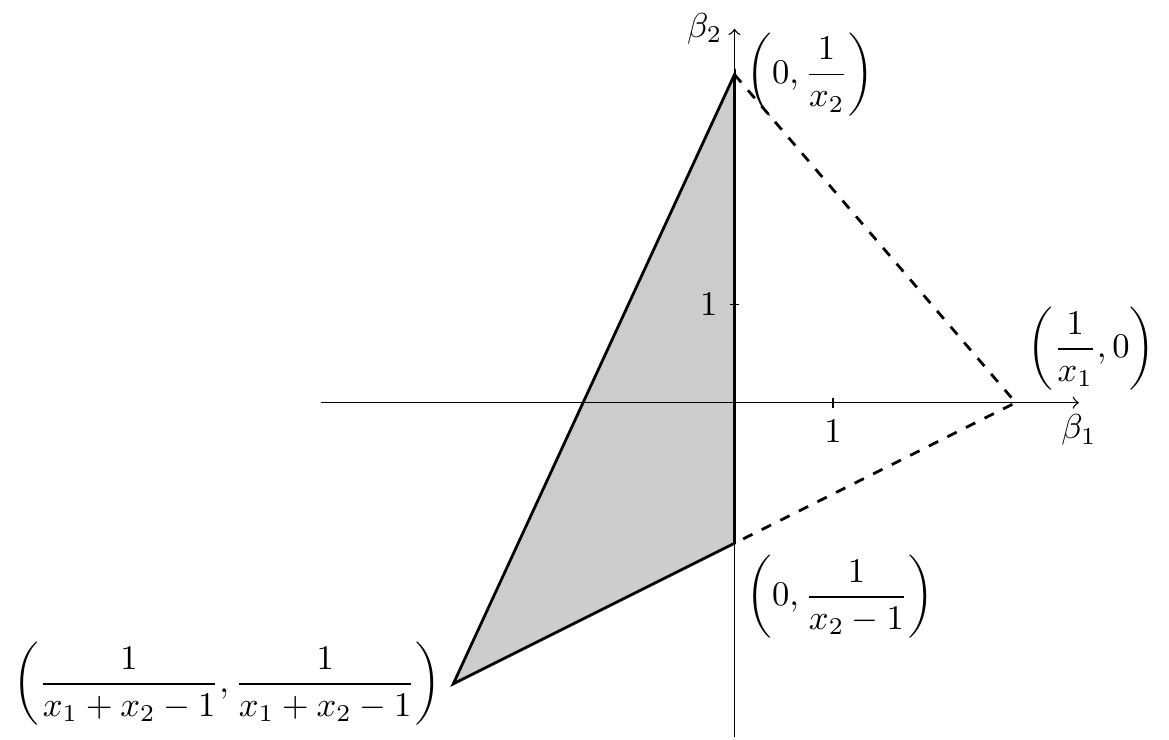}
\end{center}
Its area is
\[
\frac{1}{2}
\left(\frac{1}{x_2}-\frac{1}{x_2-1}\right)
\left(0-\frac{1}{x_1+x_2-1}\right)
=
\frac{1}{2x_2(1-x_2)(1-x_1-x_2)}
\qedhere
\]
\end{proof}

The mass of that density for the part $x_1<x_2<x_3$ is:
\[
\int_0^{\frac{1}{3}}
\int_{x_1}^{\frac{1}{2}-\frac{x_1}{2}}
\frac{1}{2x_2(1-x_2)(1-x_1-x_2)}
dx_2dx_1
=
-\frac{\pi^{2}}{24} + \frac{1}{2} \, \log\left(3\right)
\log\left(2\right) - \frac{1}{2} \, {\rm Li}_{2}(\frac{2}{3}) + \frac{3}{2} \,
{\rm Li}_{2}(\frac{1}{3}) - {\rm Li}_{2}(-\frac{1}{3})
\]
we will prove in Section~\ref{sec:remarks} that this quantity  is equal to $\frac{\pi^{2}}{24}$.



\section{Experimentations: domain of the natural
extension}\label{sec:experimentations}

Since the map $F$ is linear, the trajectory of $\lambda \bx$ is homothetic to that of $\bx$; hence, we can always normalize $\bx$ to $\frac{\bx}{\|\bx\|}$, and consider it as a point in the unit simplex. We can do the same thing for $\ba$, and both $(\bx,\ba)\in\R^d\times\R^d$ and $\widetilde F(\bx,\ba)$, its image under
$\widetilde F$, can be represented graphically by a pair of points in a pair
of unit simplexes. This gives four points in four distinct simplexes. To visualise
the effect of $\widetilde F$, we draw each of these four points with the  color
according to which element of the partition of $\Lambda$ the point $\bx$
belongs. We could do this exercice for a set of random points
$(\bx,\ba)\in\R^d\times\R^d$, but in order to find the domain of the natural extension,  it is better to consider some orbit
$(\widetilde F^n(\bx,\ba))_{n\geq0}$ for some starting point $(\bx,\ba)\in\R^d\times\R^d$.
Below are the pictures we get for some Multidimensional Continued Fraction algorithms.


\begin{figure}
\begin{center}
\includegraphics[width=\linewidth]{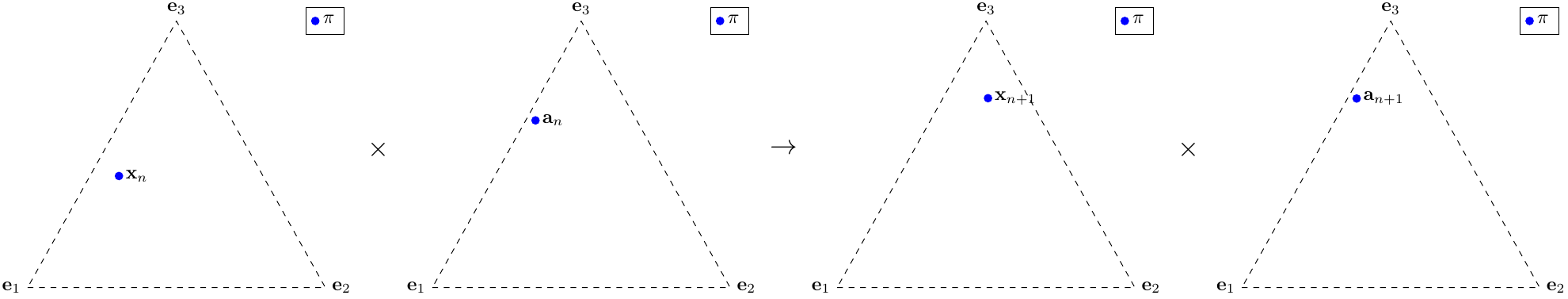}
\end{center}
\caption{For some starting point $(\bx,\ba)\in\R^d\times\R^d$, we consider
the sequence $(\bx_n,\ba_n)=F^n(\bx,\ba)$ for $n\geq0$. The points $\bx_n$,
$\ba_n$, $\bx_{n+1}$ and $\ba_{n+1}$ are drawn in four different plots with a
color or label $\pi$ associated to the matrix $M(\bx_n)$. This experimentation
allows to view the domain of the natural extension of $F$.}
\end{figure}

\begin{figure}
\begin{center}
\includegraphics[width=\linewidth]{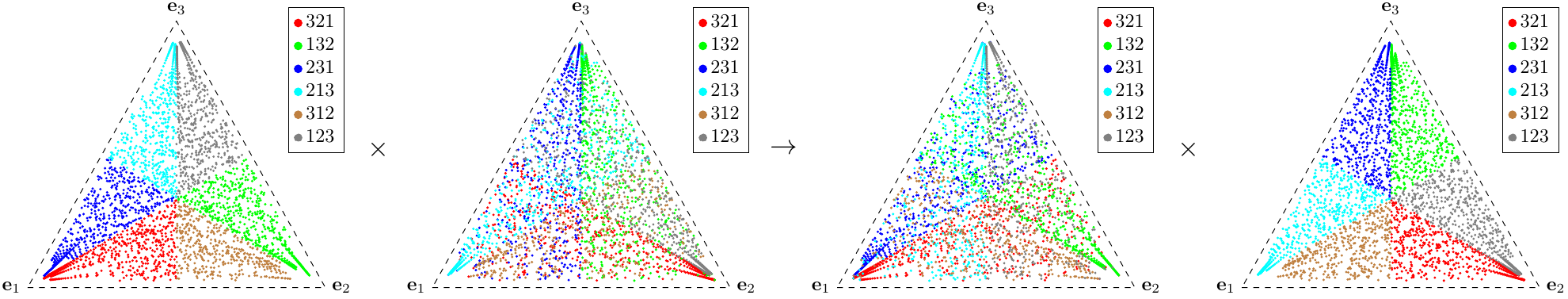}
\end{center}
\caption{Experimental natural extension of Brun algorithm.}
\end{figure}
\begin{figure}
\begin{center}
\includegraphics[width=\linewidth]{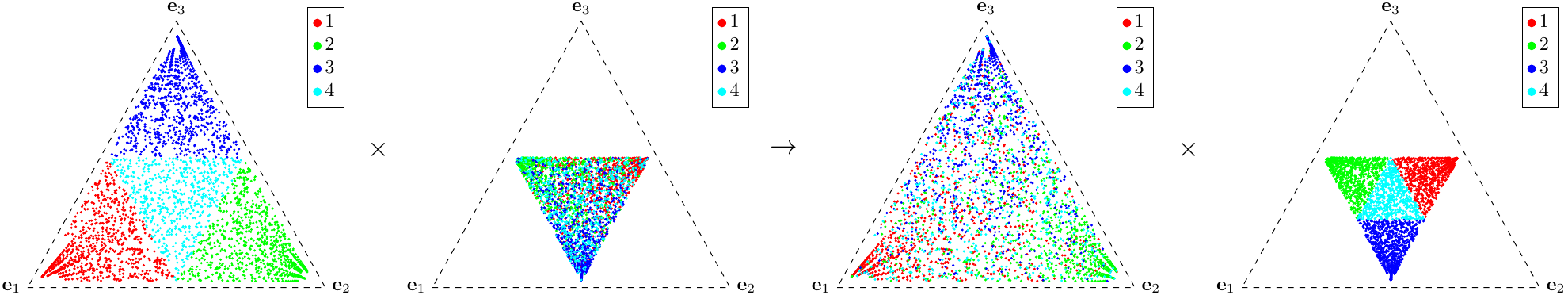}
\end{center}
\caption{Experimental natural extension of Reverse algorithm.}
\end{figure}
\begin{figure}
\begin{center}
\includegraphics[width=\linewidth]{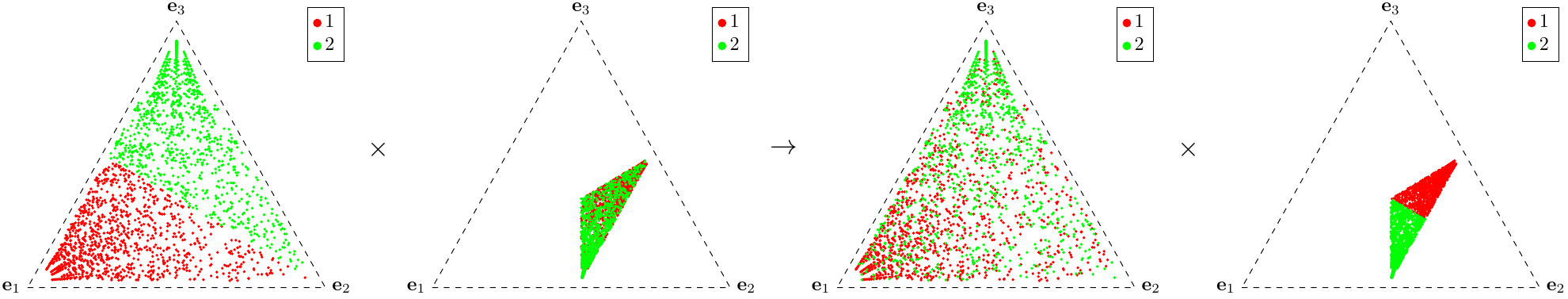}
\end{center}
\caption{Experimental natural extension of Cassaigne algorithm.}
\end{figure}
\begin{figure}
\begin{center}
\includegraphics[width=\linewidth]{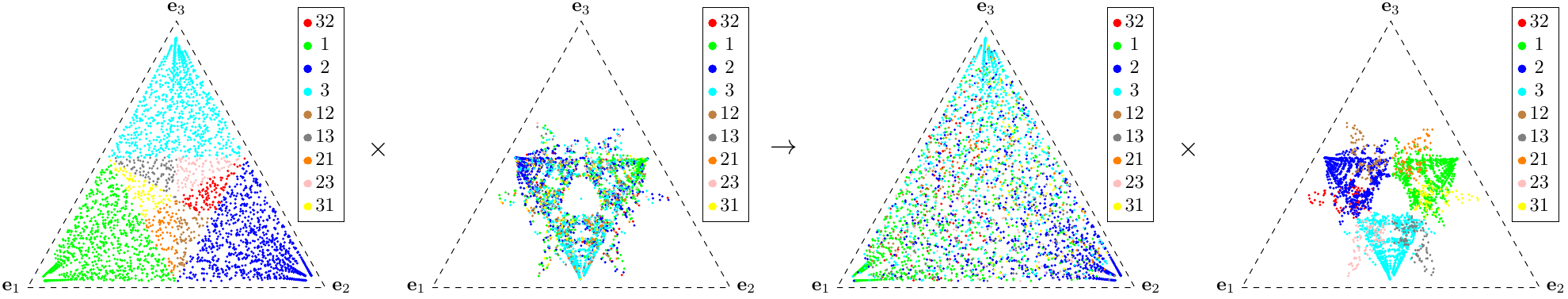}
\end{center}
\caption{Experimental natural extension of Arnoux-Rauzy-Poincaré algorithm.}
\end{figure}

The experimentation shows that a domain of positive measure can be expected
for Brun and Reverse algorithms. The method also works for Selmer algorithm.
The method does not work for Poincar\'e and Meester algorithms: both do not
have a finite measure preserving transformation. Note that Poincar\'e does
have an invariant measure equivalent to Lebesgue but it is not conservative.
Finally, the domain for AR-Poincar\'e \cite{MR3283831} and for the additive version of Jacobi-Perron algorithm are fractals and defined by an IFS. Both of these algorithms are known to have an invariant measure absolutely continuous with respect to Lebesgue measure; this was shown for Jacobi-Perron algorithm by Broise and Guivarc'h \cite{MR1838461}, and we will prove it for  AR-Poincar\'e in a forthcoming paper, but we do not know their density function.

In Figure~\ref{fig:arppng1}, we compute the first 2000000 iterations of an
orbit in a $1024\times 1024$ pixels picture. 

\begin{figure}[h]
\begin{center}
\includegraphics[width=.48\linewidth]{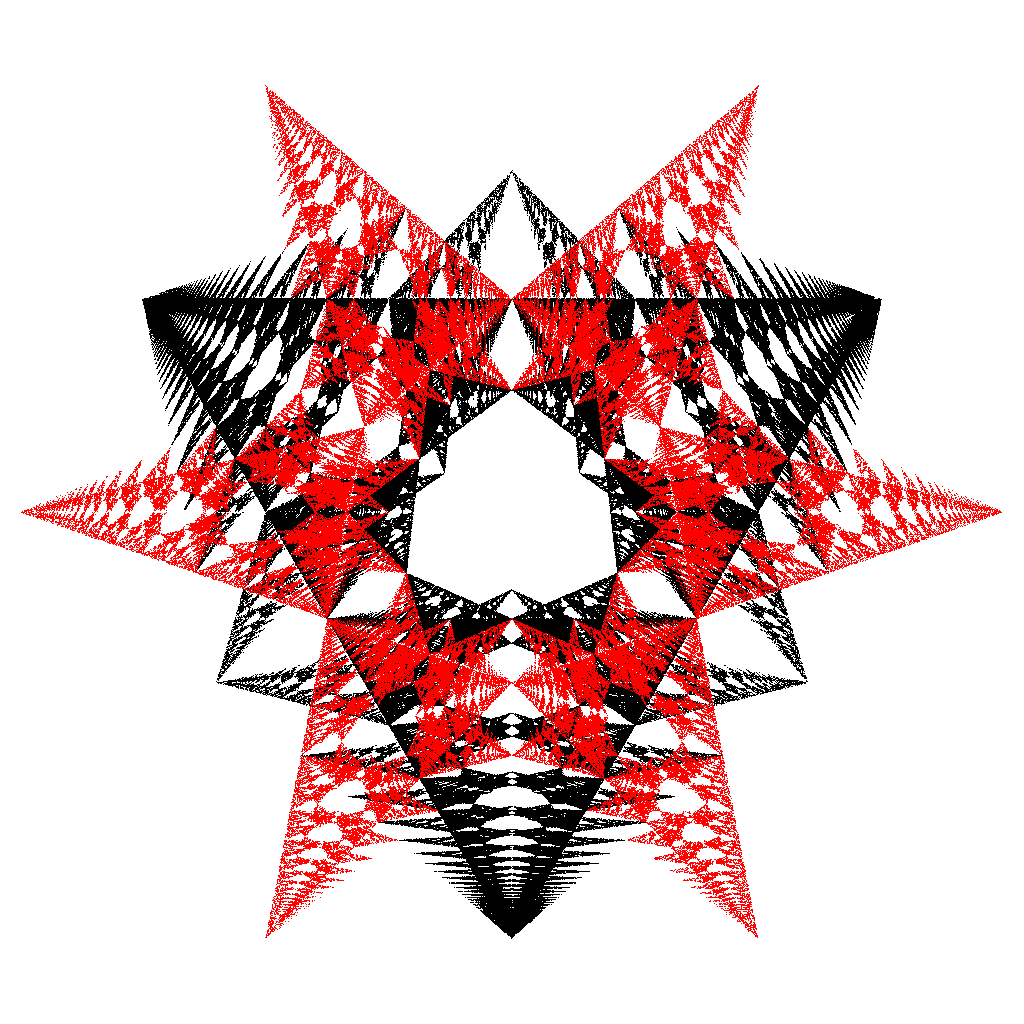}
\includegraphics[width=.48\linewidth]{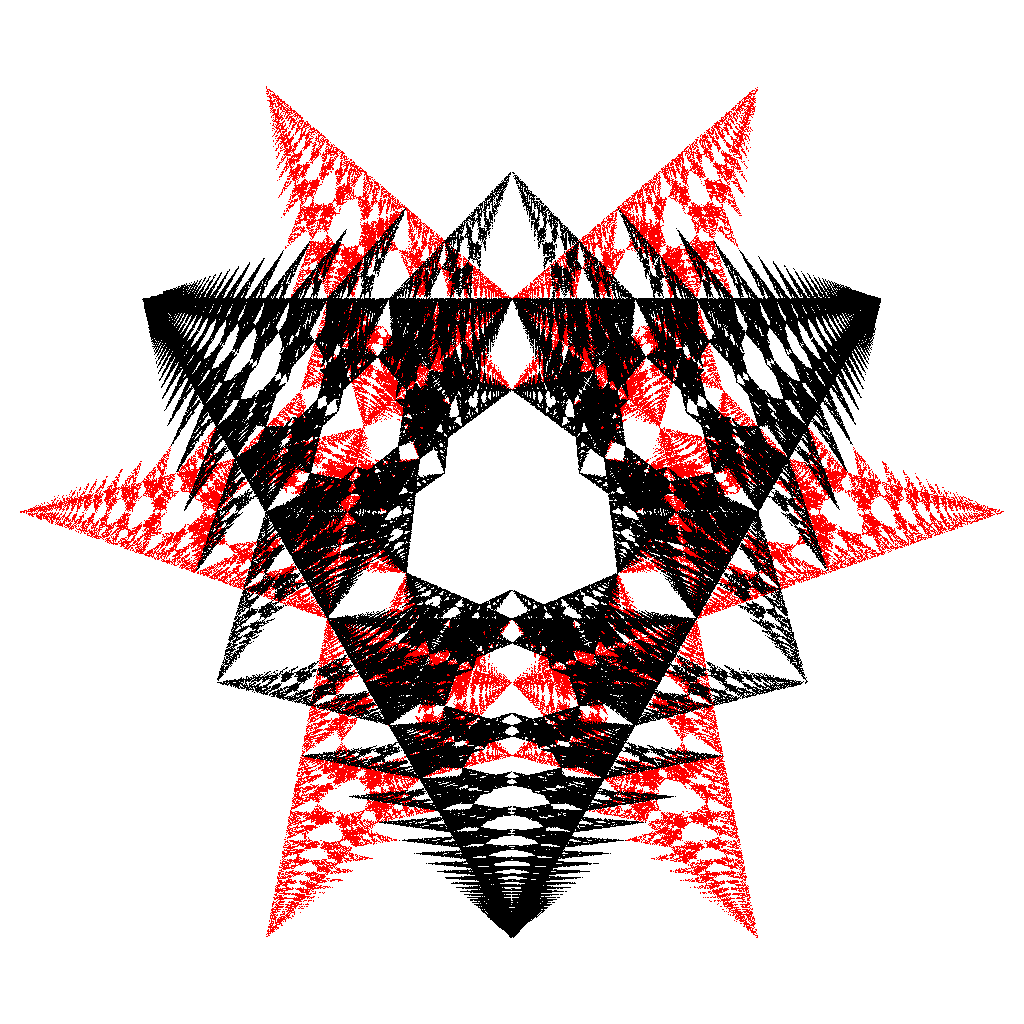}
\end{center}
\caption{$2\times10^6$ iterations shown in a $1024\times 1024$ pictures
of the window $[-0.6,0.6]\times[-0.6,0.6]$.
On the left, points in the six Poincar\'e branches are drawn last. On the
right, points in the three Arnoux-Rauzy branches are drawn last.}
\label{fig:arppng1}
\end{figure}

\begin{figure}[h]
\begin{center}
\includegraphics[width=.48\linewidth]{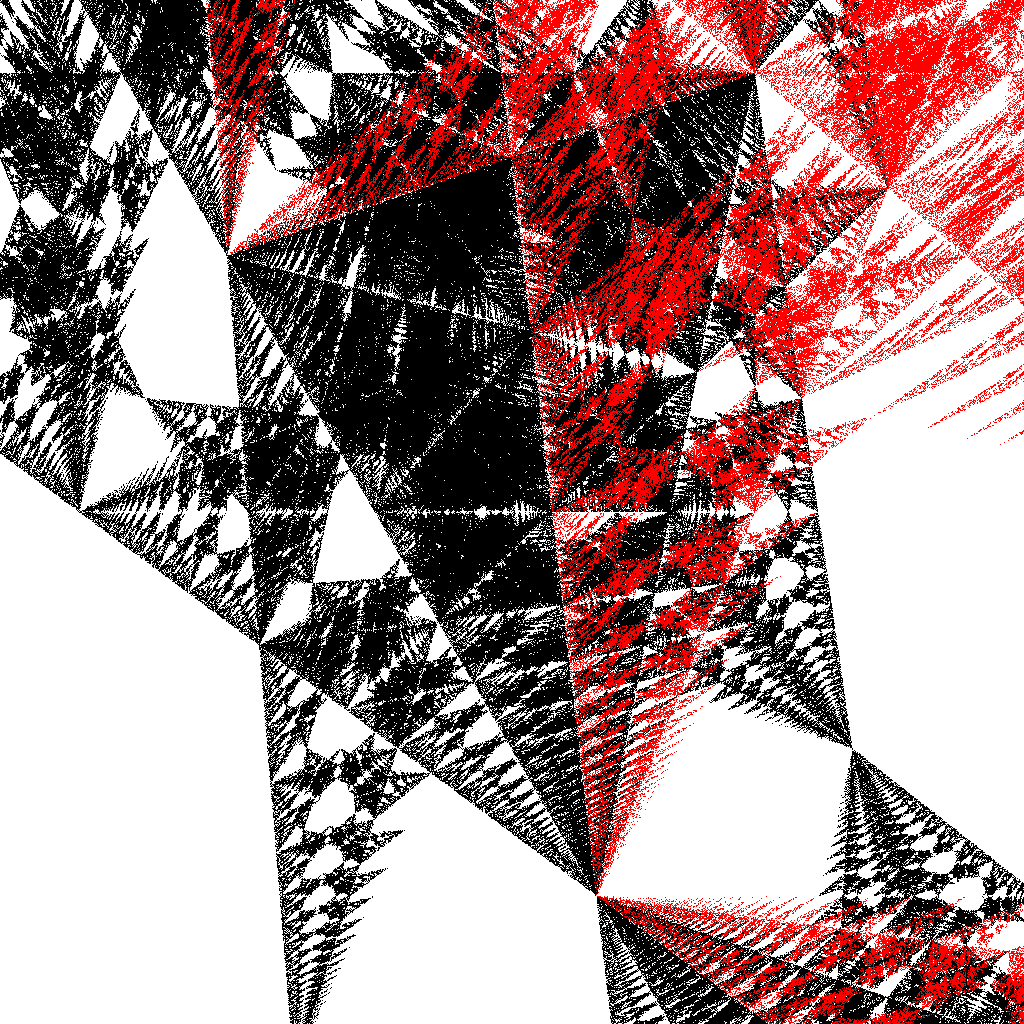}
\includegraphics[width=.48\linewidth]{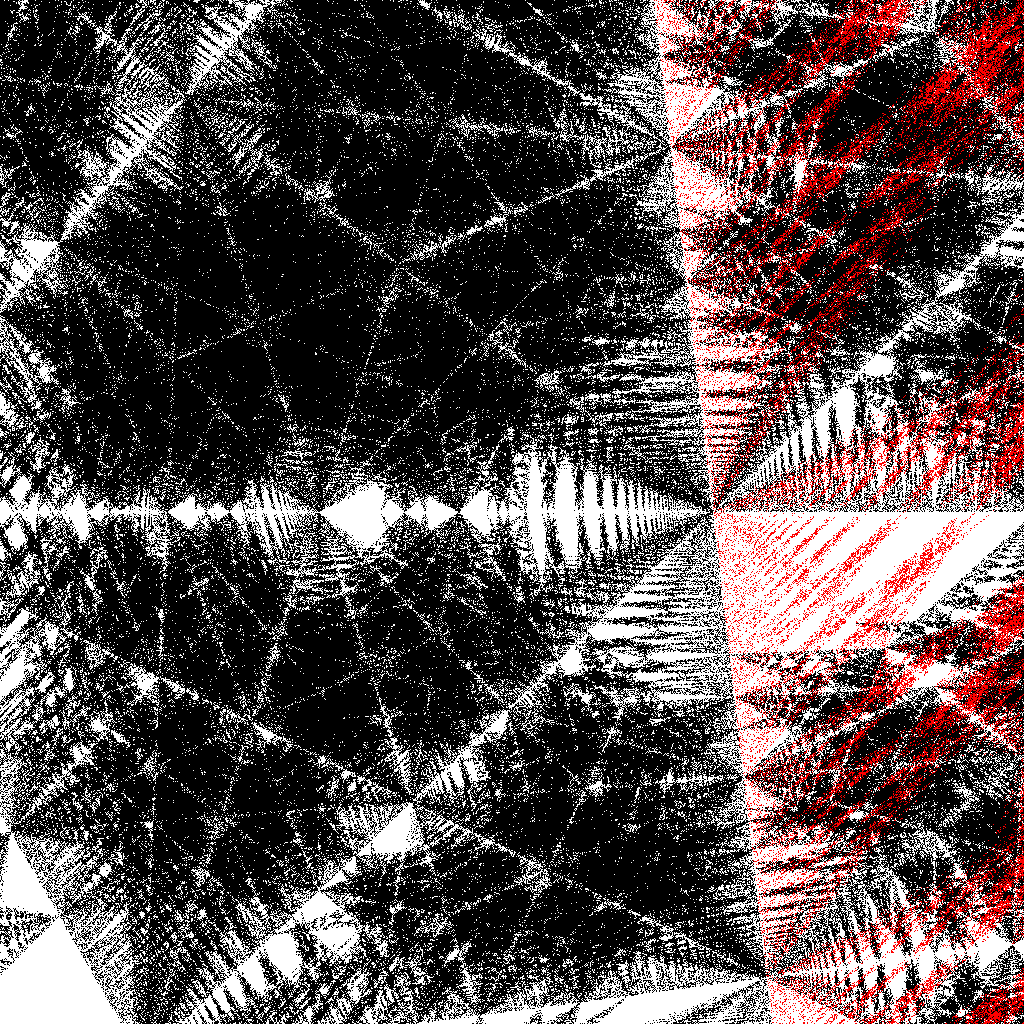}
\end{center}
\caption{
Left:$10^8$ iterations and zoom in $[0.05,0.15]\times[0.05,0.15]$.
Right:$10^9$ iterations and zoom in $[0.09,0.11]\times[0.09,0.11]$.}
\label{fig:arppng2}
\end{figure}

This set presents a symmetry of order 3; it is decomposed  in 9 subsets, and one can prove that there are 3 disjoint groups of three sets. It is unclear whether these sets contain open balls, have positive measure, or even dimension 2; it is also unclear if the three sets in each groups are disjoint, but if this is the case, it seems difficult to prove.

\section{Additional remarks}\label{sec:remarks}

\subsection{Invariant domain and fixed point theorem}

Of course, the essential element for this construction is to find  the invariant domain $D$.  It can sometimes be proved from elementary considerations, as we have seen in the two examples above, and numerical experimentations can be useful. 

When no obvious solution is found, one can use the fact that the formula for $\widetilde f$ is known. For fixed $\bx$, $\widetilde f$ is an affine map on an hyperplane in $\Lambda^*$. In some cases, it can be proved that this collection of maps depending on $\bx$ is uniformly contracting, and this implies that there exists a unique compact invariant set which is fixed by $\widetilde f$, see \cite{2015_Arnoux_Schmidt}. Numerical experiments seem to show that this property is still valid under much weaker conditions. This allows, as we have seen above, to guess the domain. 

However, this domain is interesting only if it has strictly positive measure, which is not clear in a number of cases; even in that case, it can be very difficult to describe, as shown by the case of the japanese continued fraction, which has been very precisely studied, see for example \cite{MR2946184}.

Remark also that a very well studied case, that of Rauzy induction, show the complications which can appear : if one uses exactly the procedure we describe in the present paper, one does not get a natural extension, because one gets a subset of lower dimension (given by the subspace called $H(\pi)$ by Veech, see \cite{MR644019}), and it is necessary  to add other coordinates to get the natural extension. 

\subsection{Symmetric algorithms}

The group $S_d$ of permutations on $d$ elements acts in a natural way on $\R^d$ by $\sigma.(x_1, \ldots, x_d)=(x_{\sigma(1)}, \ldots, x_{\sigma(d)})$.  We will say that a map $F$ is {\em symmetric } if, for all permutations $\sigma$,  $F(\sigma.\bx)=\sigma.F(\bx)$. Brun and Reverse algorithms are symmetric; a classical example of a non-symmetric algorithm is Jacobi-Perron algorithm: a permutation of coordinates gives in general a very different orbit.

When we have a symmetric algorithm defined on the positive cone, we can quotient by $S_d$, and restrict the algorithm to the subset  defined by $0<x_1<x_2<\cdots<x_d$. The original symmetric algorithm is often named the {\em unsorted algorithm}, and the quotient algorithm is named the {\em sorted algorithm}. These algorithms are essentially equivalent, since the sorted algorithm is a $d!$-to-1 factor of the unsorted one. The sorted algorithm is generally more convenient for explicit computations, in particular numerical experiments, since the relations between coordinates are known; but the unsorted algorithm, being more symmetric, gives usually simpler and more natural formulas, and is better adapted to theoretical considerations; another interesting feature is that, in classical examples, the unsorted algorithm preserves the orientation, but the sorted algorithm does not, since transpositions reverse the orientation. 

This leads to many apparently distinct, but essentially equivalent, presentations for the same algorithm; this is already true for the basic Farey algorithm, which takes several forms with similar properties.

\subsection{Parabolic points and acceleration}

We say that a continued fraction algorithm has an indifferent fixed point if one of the branches of the projective map $f$ associated to this algorithm  has a fixed point where the Jacobian matrix is the identity (this does not depend on the choice of coordinates). Several classical algorithms have one or several indifferent fixed points; this presents the inconvenient, in dimension 1, that the invariant measure becomes infinite, and the entropy zero.

A well-known way to get rid of this problem consists in "accelerating" the algorithm; more precisely, let $f_i$ be the branch of $f$ which has an indifferent fixed point, and let $A_i$ be its domain. For any $x$ different of the fixed point, define $n_x$ as the smallest $n$ such the $f^n(x)\notin A_i$, and replace the branch $f_i$ by the countable collection of branches $x\mapsto f^{n_x}(x)$. This replaces what is generally called an "additive" algorithm with a finite number of branches by a "multiplicative" algorithm, with a countable number of branches, no indifferent fixed point, and with a finite invariant measure.

\subsection{Changes of coordinates}

As we said before, a continued fraction map in the usual sense is a projective map; hence there is no canonical coordinates for explicit computations, and many things depend in particular on the choice of the section $\Sigma$. As an example, we show another section for Brun's algorithm, which will allow us to compute the total mass of the invariant measure in a different way.

Instead of using the norm $\sum_{i=1..3} |x_i|$, we chose the norm $\sup_{i=1..3} |x_i|$, and we work out the proof of Proposition \ref{prop:invariantmeasureBrun}
 in this setting. We suppose as before that  $\pi=(1,2,3)$. The corresponding section, for this choice of norm, is 
 
 \[
     \Sigma=
\{(x_1,x_2,x_3,\alpha_1,\alpha_2,\alpha_3)\in\R_+^{6} \mid \,
    \sum_{i=1}^{3}x_i\alpha_i=1,
    x_1<x_2<x_3=1,
     \alpha_1<\alpha_3\}.
\]

We make a different choice of coordinates here, adapted to this norm :  keep  $x_1$,  $x_2$, $\alpha_1$ and $\alpha_2$; define $\tau=-\log(|x_3|)$
and $e=\sum_{i=1}^3 \alpha_i x_i$. A straightforward computation shows that this change of coordinates has jacobian 1.

We get that $\Sigma$ is the set of 
$(x_1,x_2,\beta_1,\beta_2,\tau, e)\in\R^{6}$
    satisfying
\[
\left\{
\begin{array}{lll}
0<x_1<x_2<1,\qquad&  \tau=0,&e=1\\
0<\alpha_1<\alpha_3,&
0<\alpha_2.&.
\end{array}
\right.
\]
The equalities $e=1$ and $ x_3=1$ allows to substitute
$\alpha_3$ above, hence $\Sigma$ is also described by
\[
\left\{
\begin{array}{lll}
0<x_1<x_2<1,\qquad&  \tau=0,&e=1\\
\alpha_1(1+x_1)+\alpha_2x_2<1,& 0<\alpha_1,&0<\alpha_2.
\end{array}
\right.
\]

Hence  the density is equal to the
volume of the polytope $\{\ba: (\bx,\ba)\in\Sigma\}$:
\[
    \displaystyle
    \delta(x_1,x_2) = \int_{\{(\alpha_1,\alpha_2):
    (x_1,x_2,0,\alpha_1,\alpha_2,1)\in\Sigma\}} 1 \,d\alpha_1\,d\alpha_2
\]
This is again given by a triangle, this time with vertices $(0,0)$, $(1/(1+x_1),0)$, $(0,1/x_2)$. Hence the invariant density is given in these coordinates by $\frac 1{2x_2(1+x_1)}$.


The mass of that density on the triangle  $0<x_1<x_2<1$ is:
\[
\int_0^{1}
\int_{0}^{x_2}
\frac{1}{2x_2(1+x_1)}
dx_1dx_2
=\int_0^1\frac{\log(1+x_2)}{2x_2}\, dx_2
\]

Taking the series expansion of $\log(1+x_2)$ and integrating, we find that 
 this quantity  is equal to the series $\sum_{n=1}^{\infty}\frac{(-1)^{n+1}}{2 n^2}=\frac{\pi^2}{24}$, as announced at the end of section 5. This is just a change of variable in the integral, which can be made completely explicit. Since we made the computation on a sixth of the surface, the total mass of the invariant measure for the unsorted algorithm is $\frac{\pi^2}{4}$.
 
 The equality $-\frac{\pi^{2}}{24} + \frac{1}{2} \, \log\left(3\right)
\log\left(2\right) - \frac{1}{2} \, {\rm Li}_{2}(\frac{2}{3}) + \frac{3}{2} \,
{\rm Li}_{2}(\frac{1}{3}) - {\rm Li}_{2}(-\frac{1}{3})=\frac {\pi^{2}}{24}$ is probably well-known among the many identities for the dilogarithm function, but we could not find it in the literature.
 
 
\subsection{Variants of the classical continued fraction}

We have studied in Section~\ref{sec:farey} the unsorted Farey map. One of the advantages of this map is that it preserves orientation, and is related to matrices in $SL(2,\R)$. However, most presentations prefer to deal with the sorted map, restricting to the cone $x<y$, and in that case it is convenient to take as section $y=1$; this leads to the usual sorted Farey map, given by $x\mapsto \frac x{1-x}$ if $x<\frac 12$ and $x\mapsto \frac 1x-1$ if $x>\frac 12$. 

This map has only one indifferent fixed point, at 0. Taking the acceleration of the map at this fixed point leads to the usual Gauss map, $x\mapsto \left\{\frac 1x\right \}$.

Many other variants are possible, by taking different sections; they are all related to the geodesic flow on the modular surface and the group $SL(2,\Z)$.

\subsection{Higher dimensions}

It is natural to generalize to higher dimensions. For the reverse algorithm, a direct generalization is not possible: the basis of the algorithm is that the complement in the unit simplex of the three simplexes of size $\frac12$ given by the condition $x_i>\frac 12$ is also a simplex, which can be sent to the whole simplex by an homothety of ratio $-2$.  In dimension 4, the complement, in the unit simplex, of the four simplexes $x_i>\frac 12$ is no more a simplex, but an octahedron, so we cannot extend our simple solution to that case.

But it is possible, and easy, to generalize Brun algorithm to any dimension. We define it by subtracting the second largest coordinate by the largest coordinate. We can define an invariant domain for the natural extension in the same way. Let $\Lambda$ be the positive cone in $\R^d$, and define, for $\pi\in S_d$

\begin{align*}
\Lambda_\pi &= \{(x_1,\ldots,x_d) \in\Lambda \mid x_{\pi(1)}<x_{\pi(2)}<\cdots<x_{\pi(d)}\}\\
\Theta_\pi &= \{(x_1,\ldots,x_d) \in\Lambda \mid x_{\pi(1)}<x_{\pi(2)}<\cdots<x_{\pi(d-1)}\}\\
\Lambda^*_\pi &= \{(\alpha_1,\ldots,\alpha_d) \in\Lambda \mid \alpha_{\pi(i)}<\alpha_{\pi(d)}<\alpha_{\pi(d-1)}\text { if } i<d-1\}\\
\Theta^*_\pi &= \{(\alpha_1,\ldots,\alpha_d) \in\Lambda \mid \alpha_{\pi(i)}<\alpha_{\pi(d)}\text { if } i<d-1\}
\end{align*}

The same proof as in dimension 3, with heavier notations, proves that $\widetilde F$ sends $\Lambda_{\pi}\times \Theta^*_{\pi}$ to $\Theta_{\pi}\times \Lambda^*_{\pi}$. It is then straightforward to compute the invariant measure; similar formulas were given in \cite {MR1251147} and \cite {MR1810942}.

\begin{proposition} \label{prop:invariantmeasureBrundimd}
The density function of the invariant measure of $f:\Delta\to\Delta$ for
the $d$-dimensional Brun algorithm is
\[
    \delta(\bx) = 
    \frac{1}{(d-1)!\,x_{\pi(d-1)}}
    \sum_{A_1\subset A_2\subset\dots\subset A_{d-1}}
    \prod_{k=1}^{d-1}
    \frac{1}{1-\sum_{i\in A_k}x_i}
\]
on the part $\bx\in\Lambda_\pi$
where $A_1=\{\pi(d-1)\}$, $A_{d-1}=\{1,2,\dots,d\}\setminus\{\pi(d)\}$ and
$|A_k|=k$.
\end{proposition}

For example, taking $\pi=Id$, when $d=2$:
\[
\delta(x_1)=\frac{1}{x_{1}\,{\left(1 - x_{1}\right)}}
\]

When $d=3$:
\[
    \delta(x_1,x_2)=\frac{1}{2 \, x_{2}\, {\left(1-x_{1} - x_{2}\right)}{\left(1-x_{2}\right)}
 }
\]

When $d=4$:
\[
\delta(x_1,x_2,x_3)=
\frac{\frac{1}{{\left(1-x_{1} - x_{3}\right)} {\left(1-x_{3}\right)}
    } + \frac{1}{{\left(1 - x_{2} - x_{3}\right)} 
	{\left(1 - x_{3}\right)} }}
	{6 \,x_{3}\, {\left(1 - x_{1} - x_{2} - x_{3}\right)}}
\]

When $d=5$, $\delta(x_1,x_2,x_3,x_4)$ is
\[
\frac{\frac{\frac{1}{{\left(1 - x_{1} - x_{4}\right)} {\left(1-x_{4}
	\right)}} + \frac{1}{{\left(1 - x_{2} - x_{4}\right)}
    {\left(1 - x_{4}\right)}}}{1 - x_{1} - x_{2} - x_{4}} +
    \frac{\frac{1}{{\left(1 - x_{1} - x_{4}\right)} {\left(1 - x_{4}\right)}
	} + \frac{1}{{\left(1 - x_{3} - x_{4}\right)} {\left(1-x_{4}
    \right)}}}{1 - x_{1} - x_{3} - x_{4}} + 
    \frac{\frac{1}{{\left(1 - x_{2} - x_{4}\right)} {\left(1 - x_{4}\right)}} +
	\frac{1}{{\left(1 - x_{3} - x_{4}\right)} {\left(1 - x_{4}\right)}
}}{1 - x_{2} - x_{3} - x_{4}}}{24 \, x_{4}\, {\left(1-x_{1} - x_{2} - x_{3} - x_{4}\right)}}
\]

\begin{proof}
    In this proof, we suppose that $\pi=(1,2,\dots,d)$.
    For Brun algorithm, the surface of section for the part
    $\bx\in\Lambda_\pi$ is
\[
    \Sigma=
\{(x_1,\dots,x_d,\alpha_1,\ldots,\alpha_d)\in\R_+^{2d} \mid \,
    \sum_{i=1}^{d}x_i\alpha_i=1,
    \sum_{i=1}^{d}x_i=1,
    \bx\in\Lambda_\pi,
    \ba\in\Theta^*_\pi\}.
\]
As explained in section \ref{sec:constructing},
we use an explicit change of variables that is useful to compute explicitly
all these measures in the case of the norm $\|\bx\|=\sum_{i=1}^d |x_i|$.
Define $y_i=x_i$ for $i<d$, and $y_d=\log(\sum_{i=1}^d |x_i|)$; define
$\beta_i=\alpha_i-\alpha_d$ for $i<d$ and $\beta_d=\sum_{i=1}^d \alpha_i x_i$. A straightforward
computation shows that this change of coordinates has jacobian 1.
Below, we keep variables $x_i$ for $i<d$ instead of $y_i$.
We get that $\Sigma$ is the set of 
$(x_1,\dots,x_{d-1},y_d,\beta_1,\dots,\beta_{d-1},\beta_d)\in\R^{2d}$
    satisfying
\[
\left\{
\begin{array}{ll}
    y_d=0,&
\beta_i<0\text{ for } i<d-1,\\
\beta_d=1,&
\beta_i+\alpha_d>0\text{ for } i<d,\\
x_1<\dots<x_{d-1},\qquad& 
\alpha_d>0.
\end{array}
\right.
\]
The equalities $\beta_d=1$ and $\sum_{i=1}^d x_i=1$ allows to substitute
$\alpha_d$ above. Indeed,
\[
    1 
    = \beta_d 
    = \sum_{i=1}^d \alpha_i x_i
    = \sum_{i=1}^{d-1} (\beta_i+\alpha_d) x_i + \alpha_dx_d
    = \sum_{i=1}^{d-1} \beta_ix_i + \alpha_d\sum_{i=1}^dx_i
    = \sum_{i=1}^{d-1} \beta_ix_i + \alpha_d.
\]
So we get that $\Sigma$ is described by
\[
\left\{
\begin{array}{lll}
    y_d=0,&
\beta_i<0\text{ for } i<d-1,\\
\beta_d=1,&
\sum_{j=1}^{d-1} \beta_jx_j-\beta_i<1\text{ for } i<d,\\
x_1<\dots<x_{d-1},\qquad&
\sum_{j=1}^{d-1} \beta_jx_j<1.
\end{array}
\right.
\]
From Proposition~\ref{prop:arnouxnogueira}, the density $\delta(\bx)$ is equal
to the volume of the polytope $\{\bb: (\bx,\bb)\in\Sigma\}$:
\[
    \displaystyle
    \delta(\bx) = \int_{\{\bb: (\bx,\bb)\in\Sigma\}} 1 \,d\bb
\]
Given $\bx=(x_1,\dots,x_{d-1})\in\R^{d-1}$, the polytope corresponds to the set of
$\bb=(\beta_1,\dots,\beta_{d-1})\in\R^{d-1}$ satisfying the above inequalities.
We conclude the proof by using Lemma~\ref{lem:polytopevolume}.
\end{proof}

\begin{lemma}\label{lem:polytopevolume}
Let $I\subset\{1,2,\dots,d\}$ and $\ell\in I$.
Let $(x_1,x_2,\dots,x_d)\in\R_+^{d}$ be positive real numbers such that
$\sum_{i=1}^d x_i =1$. Let $P_{I,\ell}$ be the $|I|$-dimensional polytope
defined by the set of real vectors $(\beta_i : i\in I)$ that
satisfy:
\[
\left\{
\begin{array}{l}
    \beta_i<0\text{ for } i\in I\setminus\{\ell\},\\
\sum_{j\in I} \beta_jx_j-\beta_i<1\text{ for } i\in I,\\
\sum_{j\in I} \beta_jx_j<1.
\end{array}
\right.
\]
Then 
\[
    Vol(P_{I,\ell}) = \frac{1}{|I|(1-\sum_{i\in I} x_i)}
    \left(
	\sum_{k\in I\setminus\{\ell\}}Vol(P_{I\setminus\{k\},\ell})
    \right)
\]
\end{lemma}

\begin{proof}
    First notice that $\beta_i=\frac{1}{(1-\sum_{i\in I} x_i)}$ for all $i\in
    I$ is a solution to the second equation. 
    This is a vertex of the polytope that we denote $S$.
    The interior of the polytope can be partionned into $|I|-1$
    part $p_k$ according to the index $k\in I\setminus\{\ell\}$ for which
    $\beta_k\geq\beta_i$ for all $i\in I\setminus\{\ell\}$.
    Note that the vertex $S$ is also a vertex of each of these parts.
    In fact, we can see that each part is in fact a pyramid with top vertex
    $S$ and base given by the equation $\beta_k=0$.
    That base is a $(|I|-1)$-dimensional polytope given by the equations
\[
\left\{
\begin{array}{l}
    \beta_i<0\text{ for } i\in I\setminus\{k,\ell\},\\
    \sum_{j\in I\setminus\{k\}} \beta_jx_j-\beta_i<1\text{ for } i\in I\setminus\{k\},\\
\sum_{j\in I\setminus\{k\}} \beta_jx_j<1
\end{array}
\right.
\]
which is exactly $P_{I\setminus\{k\},\ell}$ and
whose volume is $Vol(P_{I\setminus\{k\},\ell})$.
\end{proof}

\section*{Acknowledgements}

This work was supported by the Agence Nationale de la Recherche and the
Austrian Science Fund through the projects ``Fractals and Numeration''
(ANR-12-IS01-0002, FWF I1136) and ``Dynamique des algorithmes du pgcd : une
approche Algorithmique, Analytique, Arithmétique et Symbolique (Dyna3S)''
(ANR-13-BS02-0003).
The second author is supported by a postdoctoral Marie Curie fellowship
(BeIPD-COFUND) cofunded by the European Commission.
We wish to thank Valérie Berthé for many fruitful discussions on the subject.
We are grateful to the anonymous referees for their many valuable comments. 

\bibliographystyle{alpha} 
\bibliography{biblio}

\end{document}